\documentclass[11pt]{amsart} 

\usepackage{a4wide}

\usepackage{amssymb,amsfonts,amsmath,psfrag,enumerate}
\usepackage[all]{xy}
\usepackage[T1]{fontenc} 
\usepackage[latin1]{inputenc}

\newcommand{\Vol}{\mathrm{Vol}\,}

\newcommand{\self}{\circlearrowleft}

\newcommand{\eqdef}{:=}

\newcommand{\norm}[2][{}]{\left\|#2\right\|_{#1}}

\newcommand{\pair}[2]{\left\langle #1,#2 \right\rangle}
\newcommand{\R}{\mathbf{R}}

\newcommand{\C}{\mathbf{C}}
\newcommand{\cp}{\mathbf{P}}
\newcommand{\Z}{\mathbf{Z}}
\newcommand{\N}{\mathbf{N}}

\newcommand{\supp}{\mathrm{supp}\,}

\newcommand{\cC}{\mathcal{C}}
\newcommand{\cE}{\mathcal{E}}

\newcommand{\cS}{\mathcal{S}}
\newcommand{\HR}{H_\R^{1,1}}
\newcommand{\bdd}{H^{1,1}_{bdd}}
\newcommand{\psef}{H^{1,1}_{psef}}
\newcommand{\nef}{H^{1,1}_{nef}}

\newcommand{\kod}{\operatorname{kod}}
\renewcommand{\div}{\operatorname{div}}
\newcommand{\om}{\omega}
\newcommand{\mass}[1]{\mathbf{M}[#1]}
\newcommand{\set}[1]{\left\{#1\right\}}
\newcommand{\abs}[1]{\left\vert#1\right\vert}

\newtheorem{thm}{Theorem}[section]
\newtheorem{lem}[thm]{Lemma}
\newtheorem{defn}[thm]{Definition}
\newtheorem{cor}[thm]{Corollary}
\newtheorem{rem}[thm]{Remark}
\newtheorem{eg}[thm]{Example}
\newtheorem{prop}[thm]{Proposition}

\newtheorem{theo}{Theorem}

\newcommand{\unsur}[1]{\frac{1}{#1}}
\newcommand{\rest}[1]{ \arrowvert_{#1}}
\def\1{\mathbf{1}}

\title[From cohomology to currents]{Dynamics of Meromorphic Maps with Small Topological Degree I: From Cohomology to Currents}

\date{\today}

\author{Jeffrey Diller,  Romain Dujardin and Vincent Guedj}
\address{Department of Mathematics\\
         University of Notre Dame\\
         Notre Dame, IN 46556}
\email{diller.1@nd.edu}
\address{UFR de Math\'ematiques et Institut de Math\'ematiques de Jussieu \\
         Universit\'e Paris Diderot \\
         Case 7012 \\
         2 Place Jussieu \\
         75251 Paris Cedex 05 \\
         France}
\email{dujardin@math.jussieu.fr}
\address{Universit\'e Aix-Marseille 1 \\ LATP \\ 13453 MARSEILLE Cedex 13 \\  France}
\email{guedj@cmi.univ-mrs.fr}

\thanks{First author supported by National Science Foundation grant DMS06-53678}
\subjclass{37F10, 32H50, 14E07}
\keywords{complex dynamics, meromorphic maps, complex surfaces, positive closed currents}
\begin{document}
\begin{abstract}
We consider the dynamics of a meromorphic map on a compact K\"ahler surface whose topological degree is smaller than its first dynamical degree.  The latter quantity is the exponential rate at which iterates of the map expand the cohomology class of a K\"ahler form. Our goal in this article and its sequels is to carry out a program for constructing and analyzing a natural measure of maximal entropy for each such map.  Here we take the first step, using the linear action of the map on cohomology to construct and analyze invariant currents with special geometric structure.  We also give some examples and consider in more detail the special cases where the surface is irrational or the self-intersections of the invariant currents vanish.
\end{abstract}
\maketitle

\section*{Introduction}
Throughout this paper we consider the dynamics of a meromorphic map $f:X\dashrightarrow X$ on a compact connected K\"ahler surface $X$.  Various categories of such maps have been studied from a dynamical point of view for more than twenty years now, beginning in particular with holomorphic 
self-maps \cite{FoSi} of the projective plane $\cp^2$ and polynomial automorphisms of $\C^2$ \cite{bs1, fs-henon,BLS}.  Gradually, there has emerged a clear conjectural  picture concerning the ergodic behavior of generic $f$ \cite{G5}.  The reader is referred to the surveys \cite{Sib, G6} for a more comprehensive discussion.

Though it might not be continuously defined at all points, the meromorphic map $f$ induces natural pullback and pushforward actions $f^*,f_*:H^*(X,\R)\self$ on the cohomology groups of $X$.
A well-known idea of Gromov \cite{gromov} shows that the topological entropy of $f$ is bounded above by $\lim_{n\to\infty} \frac1n \log \norm{(f^n)^*}$.  Conjecturally, equality holds. The action on cohomology can be seen as a way of keeping track of how fast the volumes of compact subvarieties are expanded by iterates of $f$.  In particular, meromorphic maps on surfaces fall into two classes: those with `large topological degree' that expand points faster, i.e. for which $f^*:H^4(X,\R)\self$ is the dominant action; and those with `small topological degree' that expand curves more quickly, i.e. for which $f^*:H^2(X, \R)\self$ predominates.  

To state the distinction more precisely, we let $\lambda_2(f)$ denote the {\em topological degree} of $f$, that is, the number of preimages of a generic point; and we let $\lambda_1(f)\eqdef \lim_{n\to\infty} \norm{(f^n)^*|_{H^{2}(X)}}^{1/n}$ denote the \emph{(first) dynamical degree}.  Then we say that $f$ has \emph{small topological degree} if $\lambda_2(f) < \lambda_1(f)$. A delicate point which must be underlined here is that on $H^2(X,\R)$, the equality $(f^n)^*=(f^*)^n$ is not true in general \cite{FoSi, Sib}.  This is due to the fact that our mappings have indeterminacy points. We say that $f$ is {\em 1-stable} if equality holds for all $n$.
 
 \medskip

The reverse of Gromov's inequality for entropy has been completely justified for maps with large topological degree \cite{briend-duval, DS, G1}. The idea is that equidistributing Dirac masses over the iterated preimages of a generic point gives rise to a convergent sequence of measures, whose limit has maximal entroy (among other good properties). For maps with small topological degree, one hopes to arrive at an interesting invariant measure by choosing two generic curves $C,C'\subset X$ and considering something like the sequence of measures
$$
\frac{f^{-n} (C)\wedge f^n(C')}{\lambda_1(f)^{2n}}.
$$
The wedge product can be understood here as a sum of Dirac masses at intersection points. Of course, the analysis and geometry of such measures is much more involved than those obtained by pulling back points.  The present work and its sequels \cite{part2, part3} are largely devoted to overcoming this extra difficulty.

\medskip

Our approach follows one used in the invertible (i.e. bimeromorphic) case.  A bimeromorphic map has small topological degree as soon as $\lambda_1 > 1$.  A broad class of such maps has been successfully analyzed (see \cite{BLS, Ca, DF, BD, Du4}) in the following fashion.

\begin{enumerate}
\item[-] Step 1: find  a birational model of $X$ where (the conjugate of)
 $f$ becomes 1-stable.
\item[-] Step 2: analyze the action on cohomology and construct a  $f^*$ (resp $f_*$) invariant and `attracting' current $T^+$ (resp.  $T^-$) with special geometric properties.
\item[-] Step 3: give a reasonable meaning to the wedge product $T^+\wedge T^-$, both from the analytic and the geometric points of view. 
This results in a positive measure $\mu$.
\item[-] Step 4: study the dynamical properties of $\mu$.
\end{enumerate}
The only step which remains incomplete in the bimeromorphic setting is Step 3.

In this paper and its sequels, we will completely  carry out Steps 2 (this paper)  and  4 \cite{part3} for arbitrary mappings of  small topological degree, and achieve Step 3 \cite{part2} for a class of meromorphic maps that goes beyond what has previously been considered even in the bimeromorphic case.
In each step, going from $\lambda_2=1$ to arbitrary $1\leq \lambda_2<\lambda_1$ brings up serious difficulties.

We stress that we will not address Step 1, which remains open in general.  Rather, we take $1$-stability as a standing hypothesis on our maps.  However, Favre and Jonsson \cite{FJ2} have recently shown that on passing to an iterate, any polynomial map of $\C^2$ with small topological degree becomes $1$-stable on some compactification of $\C^2$.  Moreover, our results in \cite{part2} suffice completely for Step 3 in the polynomial case.  Hence for polynomial maps of $\C^2$, our results and those in \cite{FJ2} can be viewed as a maximum possible generalization 
of the work completed in \cite{BLS} for polynomial automorphisms.

\medskip
\begin{center}$\diamond$\end{center}
\medskip

Let us review the results of this paper in more detail. The reader may consult \cite{part2, part3} for more about Steps 3 and 4. 

As already noted, the main purpose of this paper is to construct invariant currents and prove convergence theorems. 
To appreciate the level of generality in our results, one should note that even if we were to begin with a map of $\mathbf{P}^2$, the need for $1$-stability might lead us to a new rational surface with much more complicated geometry.  In section \ref{sec:merom} we consider in detail the spectral behavior of the action $f^*$ on $H^{1,1}_\R(X)$. It is known that when $f$ is $1$-stable  and of small topological degree, 
there is a unique (up to scale) nef class $\alpha^+\in H_{\R}^{1,1}(X)$ such that $f^*\alpha^+ = \lambda_1\alpha^+$ 
and that all other eigenvalues of $f^*$ are dominated by $\sqrt{\lambda_2}$.  

We break our first new ground by looking for a good positive current to represent $\alpha^+$. If $\alpha^+$ belongs to the interior of the nef cone, it is represented by a K\"ahler form and therefore much easier to deal with.  Finding a suitable representative for a class on the boundary of the nef cone is an important problem in complex geometry and can be quite difficult.  Demailly et al (see e.g. \cite[Section 2.5]{Dem2} and references therein) have paid much attention to this issue.  We resolve the problem for $\alpha^+$ in a fashion that is, to our knowledge, new.

\begin{theo}\label{theo:bounded} Let $f$ be a 1-stable meromorphic map of small topological degree $\lambda_2<\lambda_1$.  Then the invariant class $\alpha^+$ is represented by a positive closed $(1,1)$ current with bounded potentials. The same is true for the analogous class $\alpha^-$ invariant under $\lambda_1^{-1} f_*$.
\end{theo}
\medskip

\noindent 

Positive representatives for $\alpha^+$ and $\alpha^-$ with bounded potentials will serve as a starting point for the sequel \cite{part2} to this paper.  Here they give us a convenient way to construct the invariant currents $T^+$ and $T^-$ referred to in Step 2 from the outline above.  Actually, we prove a somewhat more general version (Theorem \ref{thm:bounded}) of Theorem \ref{theo:bounded} in which $1$-stability is unnecessary, and the hypothesis $\lambda_2<\lambda_1$ is needed only to deal with $\alpha^-$.

Section \ref{sec:T+} is devoted to constructing and analyzing the current $T^+$. Over the course of the section, we prove 

\begin{theo}
\label{main 1}
Let $f:X\self$ be a $1$-stable meromorphic map with small topological degree $\lambda_2 < \lambda_1$.
There is a positive closed $(1,1)$ current $T^+$ representing $\alpha^+$ such that for any K\"ahler form $\omega$ on $X$,
$$
\lim_{n\to\infty} \lambda_1^{-n} f^{n*} \omega = cT^+
$$
for some $c > 0$.  In particular, $f^* T^+ = \lambda_1 T^+$, and $T^+$ has minimal singularities among all such invariant currents.

If moreover $X$ is projective, then  $T^+$ is a laminar current.
\end{theo}

Versions of Theorem \ref{main 1} have been previously obtained (e.g. \cite{Sib,DF,G4,DG,Ca}) under restrictions on the surface $X$, the map $f$, or the class $\alpha^+$.  The main innovation here is that even when $\alpha^+$ is not a K\"ahler class, we recover the current $T^+$ as a limit of pullbacks of a K\"ahler form.  Our proof of laminarity in for projective $X$ depends on this.  To get the desired convergence, we work in two stages.  We first prove it when the K\"ahler form $\omega$ is replaced by the positive representative with bounded potentials from Theorem \ref{theo:bounded}.  Then we employ some delicate volume estimates from \cite{G2} and a precise understanding of the singularities of $T^+$ to get convergence for arbitrary K\"ahler forms.  

Concerning the notion of laminarity, we refer readers to \S \ref{subs:lamin} for background.  We point out that Theorem \ref{main 4} below shows that the projectivity assumption is an issue only when $X$ has Kodaira dimension zero.
\medskip

In Section \ref{sec:T-}, we consider the  pushforward operator $f_*:H^{1,1}(X)\self$.  Pushforward of $(1,1)$ currents is harder to control, but by taking advantage of the fact that $f_*$ is dual via intersection to $f^*$, we reduce some of the more difficult questions about $T^-$ to corresponding features of $T^+$.  The end result is a nearly exact analogue of Theorem \ref{main 1}. 

\begin{theo}
\label{main 2}
Let $f:X\self$ be a meromorphic map with small topological degree $\lambda_2<\lambda_1$.
There is a positive closed $(1,1)$ current $T^-$ representing $\alpha^-$ such that for any K\"ahler form $\omega$ on $X$,
$$
\lim_{n\to\infty} \lambda_1^{-n} f_*^n \omega = cT^-
$$
for some $c > 0$.  In particular, $f_* T^- = \lambda_1 T^-$, and $T^-$ has minimal singularities among all such invariant currents.

If  $X$ is projective, $T^-$ is a woven current.
\end{theo}

Currents of this sort for non-invertible maps have been considered in e.g. \cite{DTh1,fs98,G4}.
The fact that $T^-$ is woven is essentially due to Dinh \cite{Dinh}.  Wovenness is weaker than laminarity in that the curves that one averages to approximate $T^-$ are allowed to intersect each other.  This allowance is necessary for maps which are not invertible.  Another point to stress is that, while the current $T^+$ exists even for maps with large topological degree, small topological degree is essential for the construction of $T^-$.  If $\lambda_2(f) > \lambda_1(f)$, one does not generally expect that there is a single current playing the role of $T^-$.

\medskip

In section \ref{sec:examples}, we present several interesting examples of meromorphic maps with small topological degree.  A central theme of the section is that examples are plentiful on rational surfaces but much rarer on others.
In particular, by adapting arguments from \cite{Ca, DF} we classify those surfaces which admit maps with small topological degree.

\begin{theo}
\label{main 4}
 Let $X$  be a compact K\"ahler surface, supporting a meromorphic self map of small topological degree $\lambda_2<\lambda_1$.  Then either $X$ is rational or $X$ has Kodaira dimension zero.  In the latter case, by passing to a minimal model and a finite cover, one may assume that $X$ is a torus or a K3 surface and that the map is $1$-stable.
\end{theo}
 
\noindent 
We also show that invariant currents, etc. associated to maps on irrational surfaces must behave somewhat better than they do in the rational setting.

Finally, in section \ref{sec:self}, we consider maps for which the self-intersection of either $\alpha^+$ or $\alpha^-$ vanishes.  The general idea here is that such a map must be quite close to holomorphic.  

\begin{theo}
\label{main 3}
Let $f:X\self$ be a $1$-stable meromorphic map with small topological degree $\lambda_2<\lambda_1$.  If $(\alpha^-)^2$ vanishes then so does $(\alpha^+)^2$.  And if the latter vanishes, then there is modification $\pi:X\to \check X$, where $\check X$ is a (possibly singular) surface under which $f$ descends to a holomorphic map $\check f:\check X\self$. 
\end{theo}

\noindent 
This generalizes a result of \cite{DF} and, as we explain before Proposition \ref{zero}, has an interesting natural interpretation in terms of the $L^2$ Riemann-Zariski formalism developed in \cite{BFJ}.

\section{Meromorphic maps, cohomology, and positive currents}\label{sec:merom}

\subsection{Meromorphic maps}
Let $X$ be a compact K\"ahler surface with distinguished K\"ahler form $\omega_X$.  Our goal is to analyze the dynamics of a meromorphic map $f:X\dashrightarrow X$.  The term `map' is applied rather loosely here, since $f$ is technically only a correspondence.  That is, there is an irreducible subvariety $\Gamma = \Gamma_f\subset X\times X$ with projections $\pi_1,\pi_2:\Gamma\to X$, and $f=\pi_2\circ\pi_1^{-1}$.  The projection $\pi_1$ is required to be a \emph{modification} of $X$:
there is a (possibly empty) \emph{exceptional} curve $\cE_{\pi_1}\subset \Gamma$ such that $\pi_1$ maps $\cE_{\pi_1}$ onto a finite set 
of points and $\Gamma\setminus\cE_{\pi_1}$ biholomorphically onto the complement of this set.  We will require, among other things, that our map $f$ be \emph{dominating}, i.e. that the projection $\pi_2$ onto the range is surjective.  It is often advantageous, and for our purposes never a problem (see, however, the proof Theorem \ref{thm:lamin}) to replace the graph of $f$ with its minimal desingularization.  Hence we do this implicitly, assuming throughout that $\Gamma$ is smooth.

We let $I_f \eqdef \pi_1(\cE_{\pi_1})$ denote the indeterminacy locus of $f$.  The set theoretic image $f(p)$ of each $p\in I_f$  is a connected curve.  If $C\subset X$ is a curve, then we adopt the convention that $f(C) \eqdef \overline{f(C-I_f)}$ is the (reduced) `strict transform' of $C$ under $f$.  In particular $C$ belongs to the \emph{exceptional locus} $\cE_f$ if $f(C)$ is zero dimensional.  
The exceptional locus is included in turn in the {\em critical locus} $\cC_f$, which also contains curves where $f$ is ramified.
Finally, for convenience, we name the sets $I^-_f = f(\cE_f)$ and $\cE^-_f = f(I_f)$.  These are, morally speaking, the indeterminacy and exceptional loci for $f^{-1}$.

The above terminology extends trivially to the more general case of a meromorphic surface map $g:Y\dashrightarrow Z$ with inequivalent domain and range.  To the extent that iteration is not required, the discussion in the next two subsections will also apply to meromorphic maps generally.  However, for simplicity, we continue to discuss only the given map $f$.  In keeping with our abuse of the term `map' we will usually forego the more correct symbol `$\dashrightarrow$' in favor of `$\to$' when introducing meromorphic maps.

\subsection{Action on cohomology and currents}

Suppose $\theta$ is a smooth $(p,q)$ form on $X$.  We define the pullback and pushforward of $\theta$ by $f$ to be
\begin{equation}
\label{ppdef}
f^*\theta \eqdef \pi_{1*}\pi_2^*\theta,\quad f_*\theta \eqdef \pi_{2*}\pi_1^*\theta,
\end{equation}
where the action $\pi_{j*}$, $j=1,2$ is understood in the sense of currents.  Both currents $f^*\theta$ and $f_*\theta$ are actually $(p,q)$ forms with $L^1$ coefficients.   Indeed $f^*\theta$ is smooth away from $I_f$, whereas $f_*\theta$ is continuous away from 
$I^-_f$ and smooth away from $f(\cC_f)$.

\begin{defn} 
The {topological degree} of $f$
$$
\lambda_2(f) = \frac{\int f^*(\omega_X^2)}{\int \omega_X^2} = \lim_{n\to\infty} \left(\int (f^n)^*\omega_X^2\right)^{1/n}
$$ 
is the number of preimages of a generic point.  

The {first dynamical degree} of $\lambda_1(f)$ is
$$
\lambda_1(f) \eqdef \lim_{n\to\infty} \left[\int (f^n)^*\omega_X \wedge \omega_X \right]^{1/n}.
$$
We say that $f$ has {small topological degree} if $\lambda_2(f)<\lambda_1(f)$.
\end{defn}

Most often we use $\lambda_i$, $i=1,2$ as a shorthand for $\lambda_i(f)$.
Dynamical degrees are discussed at greater length in \cite{RS,DF,G1}.  As the terminology suggests, $\lambda_1(f)$ always exists and is independent of the choice of $\omega_X$.  Furthermore, it is invariant under bimeromorphic conjugacy and satisfies the inequality $1\leq \lambda_2\leq \lambda_1^2$.
Another observation is that the spectral radius of the action on $H^{0,2}$ or $H^{2,0}$ is dominated by $\sqrt{\lambda_2}$  \cite[Proposition 5.8]{Dinh}, so we could replace $H^{1,1}$ with $H^2$ in the definition of $\lambda_1$. 

Both $f^*$ and $f_*$ classically induce 
operators $f_*,f^*:H^{p,q}(X)\to H^{p,q}(X)$.  These really only interest us in two bidegrees.  When $p=q=2$, $f^*$ is just multiplication by the topological degree $\lambda_2$.  When $p=q=1$, the operators $f^*$ and $f_*$ can be quite subtle.  Both preserve the real subspace $\HR(X) \eqdef H^{1,1}(X)\cap H^2(X,\C)$, and we will generally only use their restrictions to this subspace. 
We denote by $\langle \cdot, \cdot\rangle$ the intersection (cup) product  on $H^{1,1}(X)$, and by $(\cdot)^2$ the self intersection of a class. The operators $f^*,f_*:\HR(X)\self$ are adjoint with respect to intersection.  
$$
\pair{f^*\alpha}{\beta} = \pair{\pi_2^*\alpha}{\pi_1^*\beta} = \pair{\alpha}{f_*\beta};
$$
There is also a `push-pull' formula for $f_*f^*$ \cite[Theorem 3.3]{DF}. 
The precise statement of the latter is a bit cumbersome, so 
we'll only  state  those consequences of the push-pull formula that are important to us (Propositions \ref{pushpull1} and \ref{pushpull2}).

\medskip

An important point is that pullback and pushforward might not behave well under composition.

\begin{defn}[\cite{FoSi, Sib}]
\label{as}
 We say that $f$ is 1-stable if $(f^n)^* = (f^*)^n$ for all $n\in\N$.
\end{defn}

This property is equivalent (see \cite{FoSi} or \cite[Theorem 1.14]{DF}) to the condition that $I_{f^n} \cap I^-_f = \emptyset$ for all $n\in\N$. 
If $f$ is 1-stable, then $I(f^n)=\bigcup_{j=0}^{n-1}f^{-j}I(f)$.

 It is known \cite[Theorem 0.1]{DF} that when $\lambda_2=1$, then one can always find a bimeromorphic map $\pi:\hat X \to X$ that lifts $f$ to a map 
$\hat f:\hat X\self$ that is 1-stable.  Much more recently, similar results (see below) have been obtained by Favre and Jonsson \cite[Theorem A]{FJ2} 
for a meromorphic maps obtained by compactifying polynomial maps of $\C^2$.  It remains an open problem to determine whether such results hold for arbitrary meromorphic surface maps of small topological degree. Notice that we do not use the 1-stability assumption until Section \ref{sec:T+}.

\medskip

Much of the geometry of $X$ can be described in terms of positive closed $(1,1)$ currents. Recall that the \emph{pseudoeffective} cone 
$\psef(X)\subset\HR(X)$ is the set of cohomology classes of  positive closed $(1,1)$ currents. It is `strict' in the sense that it contains no non-trivial subspaces.  The cone dual to $\psef(X)$ via intersection is $\nef(X)$, whose interior is precisely equal to the set of K\"ahler classes.  Clearly $\nef(X)\subset\psef(X)$. 

Any effective divisor $D$ on $X$ is naturally a positive closed $(1,1)$ current that acts by integration on smooth test forms.\footnote{By `divisor', in this paper, we will always mean $\R$-divisor, i.e. we allow coefficients to be real numbers rather then just integers.} Most often we use the same letter for a curve and the associated reduced effective divisor, for a divisor and the associated current of integration, and for a current and its cohomology class.  The  context should make the point of view clear in each instance.

\medskip

Given any positive closed $(1,1)$ current $T$ on $X$, we may write $T = \theta + dd^c u$ where $\theta$ is a smooth closed $(1,1)$ form cohomologous to $T$ and $u$ is a $\theta$-plurisubharmonic function determined up to an additive constant by $T$ and $\theta$.  We call $u$ a \emph{potential} for $T$ relative to $\theta$.  The definitions of pushforward and pullback given in \eqref{ppdef} may be applied to $T$, once we declare that $\pi_j^* T \eqdef \pi_j^*\theta + dd^c u\circ \pi_j$ for either projection $\pi_j:\Gamma\to X$.  Thus defined, $f^*T$ and $f_*T$ are positive closed $(1,1)$ currents that vary continuously with $T$ in the weak topology on currents. In particular, they do not depend on the choice of $\theta$ and $u$. 

It is immediate from adjointness that $f^*$ and $f_*$ preserve $\psef(X)$ and $\nef(X)$.

The following  consequence of the pushpull formula from \cite{DF} will be important to us.
Notice that there is a similar statement for cohomology classes rather than currents.
\begin{prop}
\label{pushpull2}
For any positive closed $(1,1)$ current $T$, we have
$$
f_*f^* T = \lambda_2(f) T + \sum \pair{T}{F_i} F_i
$$
where $F_i$ is an effective divisor supported on $\cE^-_f$.  

Furthermore, from the precise expression of the $F_i$, given $p\in I(f)$, if all intersections $\pair{T}{C}$ with curves $C\subset f(p)$
are non-negative, then $E^-(T)|_{f(p)}$ is effective.  If, additionally, one such intersection is positive, then $E^-(T)$ charges all of $f(p)$.
\end{prop}

It is useful to know how $f^*$, $f_*$ act on curves.

\begin{prop} 
\label{curves}
Suppose that $C\subset X$ is an irreducible curve.  Then $f^*C = \sum \mu_j C_j + D$, where $f(C_j) = C$, $\mu_j$ is the local degree of $f$ near a generic point of $C_j$, and $D$ is an effective divisor with support exactly equal to those curves in $\cE_f$ that map to $C$. On the other hand $f_*C = (\deg f|_C) f(C) + D$, where $D$ is an effective divisor with $\supp D = f(I_f\cap D)$.
\end{prop}

\subsection{Spectral analysis of $f^*$ and $f_*$}
Let $\norm{\cdot}$ be any norm on $\HR(X)$, and let
$$
r_1(f) \eqdef \lim_{n\to\infty} \norm{(f^*)^n}^{1/n}
$$
be the spectral radius of $f^*$. In general $r_1(f)\geq \lambda_1(f)$ with equality if $f$ is 1-stable.

\begin{thm}[\cite{DF}]
\label{spectral}  Suppose $r_1(f)^2 > \lambda_2(f)$.
Then $r_1(f)$ is a simple root of the characteristic polynomial of $f^*$ (resp $f_*$), and the corresponding eigenspace is generated by a nef class $\alpha^+$ (resp $\alpha^-$), and $\pair{\alpha^+}{\alpha^-}>0$.
The subspace  $(\alpha^-)^\perp \eqdef \{\beta\in\HR(X):\pair{\beta}{\alpha^-} = 0\}$ is the unique $f^*$ invariant subspace complementary to $\R\alpha^+$, and there is a constant $C>0$ such that  for every $\beta\in  (\alpha^-)^\perp$
we have
$$
\norm{(f^*)^n\beta} \leq C\lambda_2^{n/2}\norm{\beta} \text{ for all } n\in\N. 
$$
The corresponding result holds for $f_*$.
\end{thm}

For convenience, we normalize the invariant classes $\alpha^+, \alpha^- \in \nef(X)$ and the distinguished K\"ahler form $\omega_X$ so that $\pair{\alpha^+}{\alpha^-} = \pair{\alpha^+}{\omega_X} = \pair{\alpha^-}{\omega_X} = 1$.  
This is essentially \cite[Theorem 5.1]{DF}, where it is shown that $r_1(f)$ is a simple root of the characteristic polynomial and that all other roots have magnitude no greater than $\sqrt{\lambda_2}$.  It suffices for establishing Theorem \ref{spectral} to show further that  eigenspaces
 associated to  a root with magnitude equal to $\sqrt{\lambda_2}$ is generated by eigenvectors.  
The arguments from \cite{DF} are easily modified to do this.  An alternative approach to the second assertion in  the theorem  may be found in the more recent paper \cite{BFJ}, where it is shown that we can bypass  1-stability to  obtain interesting information about the cohomological behaviour of meromorphic maps.

\subsection{Positive currents with bounded potentials}\label{subs:bdd}
We now prove Theorem \ref{theo:bounded} in the following slightly more general form:

\begin{thm}\label{thm:bounded} Let $f$ be a meromorphic map such that  $\lambda_2(f) < r_1(f)^2$.
Then the invariant class $\alpha^+$ is represented by positive closed $(1,1)$ currents with bounded potential.  If $f$ has small topological degree then the same is true of $\alpha^-$.
\end{thm}

The remainder of this subsection is devoted to the proof. 

\begin{lem}
\label{negative1}
Let $Y,Z$ be compact complex surfaces and $\pi:Y\to Z$ be a proper modification. Let $\eta$ be a smooth closed $(1,1)$ form such that $\pair{\eta}{C}\geq 0$ for every curve $C\subset\cE_\pi$.  Then potentials for $\pi_*\eta$ are bounded above.
\end{lem}

\begin{proof}(see also the proof of \cite[Theorem 2.4]{DG})
We write $\pi_*\eta = \eta' + dd^c u$ for $\eta'$ a smooth closed $(1,1)$ form and $u\in L^1(X)$.  By hypothesis and Proposition \ref{pushpull2}  applied to $\pi^{-1}$, we have
$$
\pi^*\eta' + dd^c u\circ\pi = \pi^*\pi_* \eta = \eta + D,
$$
where $D$ is an effective divisor.  Thus $u\circ \pi$ is quasiplurisubharmonic and (in particular) bounded above on $Y$.  It follows that $u$ is bounded above on $Z$.
\end{proof}

\begin{lem}
\label{negative2}
Let $\theta$ be a smooth closed $(1,1)$ form on $X$ such that $\pair{\theta}{C} \geq 0$ for every curve $C\subset \cE_f$.
Then any potential for $f_*\theta$ is bounded above.  Similarly, if $\pair{\theta}{C} \geq 0$ for every curve $C\subset \cE_f^-$, then any potential for $f^*\theta$ is bounded above.
\end{lem}

\begin{proof}
Consider first $f^*\theta = \pi_{1*}\pi_2^*\theta$.  For each irreducible $C\subset\cE_{\pi_1}$, we have that $\pi_2(C)$ is either trivial or an irreducible curve in $\cE_f^-$.  Hence $\pair{\pi_2^*\theta}{C}\geq 0$.  The assertion thus follows from
 Lemma \ref{negative1} applied to  $\pi = \pi_1$ and $\eta=\pi_2^*\theta$.

Now consider $f_*\theta$.  We recall (see e.g. the paragraph before Lemma 2.4 in \cite{BFJ}) that there exists a modification $\pi:Y\to X$ that lifts $f = \pi\circ h$ to a meromorphic map $h:X\to Y$ with $\cE_h = \emptyset$ and that $f_*\theta = \pi_*h_*\theta$.  We claim that $h_*\theta = \theta' + dd^c u$, where $u$ is a bounded function and $\theta'$ is a smooth form satisfying $\pair{\theta'}{C} \geq 0$ for all $C\subset\cE_\pi$.
Given the claim, we can apply Lemma \ref{negative1} with $\eta = \theta'$, obtaining
$$
f_*\theta = \pi_*\theta' + dd^c u\circ \pi^{-1}.
$$
Hence $f_*\theta$ has potentials that are bounded above.

It remains to prove the claim.  Let $\Gamma_h$ be the minimal desingularization of the graph of $h$, and $\pi_X:\Gamma_h\to X$, $\pi_Y:\Gamma_h\to Y$ be projections onto domain and range.  Since $h$ collapses no curves, we have $\cE_{\pi_Y}\subset\cE_{\pi_X}$.  In particular, for each connected component $C\subset\cE_{\pi_Y}$, the image $\pi_X(C) = p$ is a point.  We write 
$\theta = dd^c \psi$ on a neighborhood $U\ni p$ and obtain that $\pi_X^*\theta = dd^c \psi\circ \pi_X$ is $dd^c$-exact on a neighborhood of $C$.  Therefore, if $q\in Y$ is any point---even a point in the image of $\cE_{\pi_Y}$, there is a neighborhood $V_q\ni q$ such that $\pi_X^*\theta = dd^c \varphi_j$ is $dd^c$-exact on each connected component $V_j$ of $\pi_Y^{-1}(V_q)$.  This gives us that
$$
h_*\theta = \pi_{Y*} \pi_X^*\theta = dd^c \sum_j \pi_{Y*} \varphi_j
$$
has bounded potentials near $q$.  Since $q$ is arbitrary, the claim is established.
\end{proof}

Now let $h^{1,1} = \dim\HR(X)$, and fix smooth closed $(1,1)$ forms $\theta_1,\dots,\theta_{h^{1,1}}$ whose cohomology classes form a basis for $\HR(X)$.  Then for each positive closed $(1,1)$ current $T$ on $X$, we have a unique decomposition 
$$
T = \theta_T + dd^c V_T
$$
where $\theta_T \in \Theta \eqdef \oplus_{j=1}^n \R\,\theta_j$ and $V_T \in L^1_0(X)\eqdef \{\psi\in L^1(X):\int \psi\,\omega_X^2 = 0\}$.
Using the weak topology on the set of positive closed $(1,1)$ currents, we have that both $\theta_T$ and $V_T$ depend continuously on $T$.  As the dependence is also linear, the decomposition extends naturally to any difference $T_1-T_2$ of positive closed $(1,1)$ currents.  In particular, it extends to all smooth closed $(1,1)$ forms on $X$ and to their images under pushforward and pullback by meromorphic maps.

We give $H_\R^{1,1}(X)$ the norm $\norm[H^{1,1}]{\sum c_j\theta_j} \eqdef \max |c_j|$.  The following is essentially a restatement of \cite[Lemma 2.2]{BD}.

\begin{prop}
\label{uniform upper bound}
There is a constant $M$ such that $V_{f^*\theta},V_{f_*\theta} \leq M\norm[H^{1,1}]{\theta}$ for every $\theta\in \Theta$ representing a nef class.
\end{prop}

The difficult point here is that the form $\theta$ is not itself positive.  So despite the positivity of the class and the normalization of potentials, we cannot directly apply compactness theorems for positive closed $(1,1)$ currents.

\begin{proof}
We work only with pullbacks, the proof being identical for pushforwards.  Let $H \eqdef \{\theta\in\Theta: \pair{\theta}{\omega_X} = 1\}$ and $N = \{\theta\in\Theta:\theta \text{ represents a nef class}\}$.  Then $N\cap H$ is a compact convex subset of $\Theta$ that avoids $0$.  Since any $\theta\in\Theta$ representing a nef class may be rescaled to give an element in $N\cap H$, it suffices to find $M$ satisfying $M\geq V_{f^*\theta}$ for all $\theta\in N\cap H$.

Let $\tilde N = \{\theta\in\Theta:\pair{\theta}{C} \geq 0\text{ for every irreducible } C\subset\cE^-_f\}$.  Then $\tilde N$ is defined by finitely many linear inequalities and contains $N$.  Hence we can find \emph{finitely many} elements $\eta_1,\dots,\eta_m\in \tilde N\cap H$ whose (compact) convex hull contains $N\cap H$.  By Lemma \ref{negative2}, we have $M$ such that 
$V_{f^*\eta_j} \leq M$ for $j=1,\dots,m$.  Since the function $\theta \mapsto \sup V_{f^*\theta}$ is convex on $\tilde N\cap H$, we have $V_{f^*\theta}\leq M$ for every $\theta$ in the convex hull of $\eta_1,\dots,\eta_m$.
\end{proof}

For any class $\alpha\in\psef(X)$, we set
$$
\norm[bdd]{\alpha} = \inf
\{\sup V_T - \inf V_T:T\geq 0\text{ represents } \alpha \}\leq \infty, 
$$
and we let 
$$
\bdd(X) \eqdef \{\norm[bdd]{\alpha}<\infty\}
$$ 
be the convex cone of classes represented by positive closed currents with bounded potentials.  While $\norm[bdd]{\alpha}$ depends on our choice of $\Theta$, the cone $\bdd(X)$ does not. To our knowledge, this cone has not been previously considered.  

\begin{prop}
\label{whereisbdd}
For any K\"ahler surface $X$, we have $H^{1,1}_{kahler}(X)\subset \bdd(X) \subset \nef(X)$.  There exist $X$ for which both inclusions are strict.  Hence $\bdd(X)$ is neither open nor closed in general.
\end{prop}

\begin{proof}
K\"ahler forms have smooth local potentials, so K\"ahler classes belong to $\bdd(X)$ by definition.  On the other hand, if $V_T$ is bounded for a given $T$, then it is well-known \cite{BT} that $T\wedge S$ is a well-defined positive measure for any other positive closed current on $X$.  In particular $\pair{T}{S} \geq 0$,
which implies that $T$ represents a nef class.

Finally, \cite[Example 1.7]{DPS} exhibits a $\cp^1$ bundle $X$ over an elliptic curve $C$ for which $\bdd(X)\neq \nef(X)$.  Moreover, the pullback to $X$ of any K\"ahler form on $C$ is smooth and positive and represents a class with zero self-intersection.  This shows that $\bdd(X)$ is larger than the interior of $\nef(X)$.
\end{proof}

\begin{thm} 
\label{bounded to bounded}
There is a constant $C>0$ such that
$$
\norm[bdd]{f^*\alpha} \leq \norm[bdd]{\alpha} + C\norm[H^{1,1}]{\alpha},\quad
\norm[bdd]{f_*\alpha} \leq \lambda_2\norm[bdd]{\alpha} + C\norm[H^{1,1}]{\alpha}.
$$
Thus $f^*$ and $f_*$ preserve $\bdd(X)$.
\end{thm}

\begin{proof}
We deal only with $f^*\alpha$.  The only difference in the pushforward case comes from the fact that for functions
$V$ bounded above on $X$, one has $\sup f_*V \leq \lambda_2 \sup V$.  
Let $T$ be a positive closed current representing $\alpha$ such that $\sup V_T - \inf V_T < \infty$.  Then
$$
f^* T = f^*\theta_T + dd^c V_T\circ f = \theta_{f^* T} + dd^c (V_{f^*\theta_T}  + V_T\circ f).
$$
Note that $V_{f^*\theta_T}$ is smooth off $I_f$.

Let $U \subset\subset U'\subset X$ be open neighborhoods of $I_f$ small enough that each form $\theta_j$ can be expressed as $dd^c \rho_j$ for some smooth $\rho_j:U'\to [0,1]$.  Writing $\theta_{f^*T} = \sum c_j\theta_j$, we let $\rho\eqdef \sum c_j\rho_j$.  Then $V_T\circ f + V_{f^*\theta_T} + \rho$ is a potential for $f^*T$ on $U'$.  So for $R>0$ large enough, the function 
$$
u := \left\{\begin{array}{lcl} V_T\circ f + V_{f^*\theta_T} &\text{on} & X-U \\
                               \max\{V_T\circ f + V_{f^*\theta_T},-R-\rho\} & \text{on} & U'
            \end{array}\right.
$$
is well-defined and bounded.  Indeed, paying more careful attention, one finds that
$$
-R = \inf_X V_T + \inf_{X\setminus U} V_{f^*\theta_T} + \inf_{U'}\rho.
$$
suffices here.  The current $S := \theta_{f^*T} + dd^c u$ represents $f^*\alpha$ and agrees with $f^*T$ outside $U$.  Since $\max\{V_T\circ f + V_{f^*\theta_T} + \rho, -R\}$ is a potential for $S$ on $U'$, we see that $S\geq 0$ on all of $X$.  Hence $f^*\alpha \in H^{1,1}_{bdd}(X)$, with $\norm[bdd]{f^*\alpha} \leq \sup_X u - \inf_X u$.

Now 
$$
\sup_X u \leq \sup_X V_T + \sup_X V_{f^*\theta_T} + \sup_{U'} |\rho| \leq \sup_X V_T + C\norm[H^{1,1}]{\alpha},
$$
where $\sup_X V_{f^*\theta_T}$ is controlled by Propositions \ref{uniform upper bound} and \ref{whereisbdd}, and 
$\sup_{U'} |\rho|$ is controlled by the facts that $0\leq \rho_j\leq 1$ and 
$|c_j| \leq \norm[H^{1,1}]{f^*\alpha} \leq C\norm[H^{1,1}]{\alpha}$.

In the other direction, our choice of $R$ gives
$$
\inf_X u = \inf_{U'} u \geq \inf_{U'} (-R - \rho) \geq \inf_X V_T + \inf_{X\setminus U} V_{f^*\theta_T} - 2\sup_{U'}|\rho|.
$$
The final term is estimated as above.  Writing $\theta_T = \sum_j b_j \theta_j$, we  
control the middle term by 
$
\inf_{X\setminus U} V_{f^*\theta_T} \geq -\sum_j |b_j|\max_{X\setminus U} |V_{f^*\theta_j}| 
                                      \geq -C\norm[H^{1,1}]{\alpha}.
$
Thus we arrive at
$$
\norm[bdd]{f^*\alpha} \leq \sup_X u - \inf_X u \leq (\sup_X V_T-\inf_X V_T) + C\norm[H^{1,1}]{\alpha}.
$$
The proof is ended by taking the infimum of the right side over all $T\geq 0$ representing $\alpha$.
\end{proof}

\begin{proof}[Proof of Theorem \ref{thm:bounded}]
Recall that we have normalized so that $\frac{(f^*)^n\omega_X}{r_1(f)^n}$ tends toward $\theta^+$ in cohomology.
From Theorem \ref{bounded to bounded}, one has 
$$
\norm[bdd]{(f^*)^n\omega_X} \leq \norm[bdd]{\omega_X} + C\sum_{j=0}^{n-1}\norm[H^{1,1}]{(f^*)^j\omega_X}< M r_1(f)^n
$$
for some $M$ independent of $n$.  Dividing through by $r_1(f)^n$ and appealing to compactness of the set of positive closed currents $T = \theta_T + dd^c V_T$ with $|V_T| \leq  M$, we conclude that $\theta^+ \in H^{1,1}_{bdd}(X)$.  The proof for $\alpha^-$ is similar.
\end{proof}

\section{The canonical $f^*$-invariant current}\label{sec:T+}
We now construct and analyze the invariant current $T^+$.  There are of course many precedents (see e.g. \cite{Sib, Fav2, DG}) for this.  The novelty here concerns the level of generality in which we are working. 

\subsection{Construction of $T^+$}\label{subs:cv+}

Recall from Theorems \ref{spectral} and \ref{thm:bounded} that when $f$ is $1$-stable and $\lambda_1^2 > \lambda_2$, there is a unique (normalized) class $\alpha^+\in H_{bdd}^{1,1}(X)$ such that $f^*\alpha^+=\lambda_1\alpha^+$.

\begin{thm}\label{thm:cv+}
Suppose that $f$ is $1$-stable and that $\lambda_1^2 > \lambda_2$.  Then there is a positive closed $(1,1)$ current $T^+$ representing $\alpha^+$ such that $f^*T^+ = \lambda_1 T^+$ and for any smooth form $\theta^+$ representing $\alpha^+$, we have weak convergence
$$
\lim_{n\to\infty} \lambda_1^{-n}f^{n*}\theta^+ \to T^+.
$$
The latter holds more generally for (non-smooth) representatives with bounded local potentials. 
\end{thm}

This theorem is proven with a different argument in \cite{DG}.  Here we give only those details of the proof that are different and/or important for the sequel.  An advantage to the present approach is that it works equally well for pushforwards (see Theorem \ref{thm:cv-}).

\begin{proof}
By the $dd^c$-lemma, $\lambda_1^{-1} f^* \theta^+=\theta^+ +dd^c \gamma^+$,
where $\gamma^+ \in L^1(X)$ is uniquely determined by the normalization $\int_X \gamma^+ \om_X^2=0$.
We pull this equation back by $f^{n-1}$ and get 
\begin{equation}\label{eq:gn+}
\frac{(f^n)^* \theta^+}{\lambda_1^n}=\theta^++dd^c g_n^+,\;
\text{ where } g_n^+=\sum_{j=0}^{n-1} \frac{1}{\lambda_1^j} \gamma^+ \circ f^j.
\end{equation}
We claim that the sequence $(g_n^+)$ converges.  The main point is that $\gamma^+$ is bounded above, so that the sequence is essentially decreasing.  Given the claim, convergence follows from a (by now standard) argument of Sibony \cite{Sib}, whose
details we omit.  On the level of currents, we obtain $\lim_{n\to\infty} \lambda_1^{-n} f^{n*} \theta^+ 
= T^+$, where 
$
T^+ \eqdef \theta^+ + dd^c g^+,
$
is a priori a difference of positive  closed $(1,1)$ current, and represents $\alpha^+$.

To prove the claim, we apply Theorem \ref{thm:bounded} to get a positive representative $\om^+ =\theta^+ +dd^c u \geq 0$ for $\alpha^+$ with potential $u \in L^{\infty}(X)$.  Thus
$$
\unsur{\lambda_1}f^*\om^+ = \theta^+ + dd^c \left(\gamma^+ + \unsur{\lambda_1} u\circ f \right).
$$
Since $f^*\om^+$ is positive, it follows that $\gamma^+ + \lambda_1^{-1} u\circ f $ is bounded above.  Since $u$ is bounded, we conclude that $\gamma^+$ itself is bounded above.

Furthermore, we see that 
$$
\lim_{n\to\infty} \lambda_1^{-n} f^{n*}\omega^+ = \lim_{n\to\infty} \lambda_1^{-n}(f^{n*}\theta^+ + dd^c u\circ f^n) = T^+ + dd^c 0.
$$
from which we infer that $T^+$ is positive.  From continuity of $f^*$ on positive closed $(1,1)$ currents, we finally conclude that 
 $f^* T^+ = \lambda T^+$.
\end{proof}

\begin{rem}
\label{minsing}
 It easily follows from the second part of the proof that $T^+$ has {\em minimal singularities among invariant currents}: that is, let $S$ be a positive closed current satisfying $f^*S=\lambda_1 S$,  rescaled so that
$S$ is cohomologous to $\alpha^+$.   Hence $S = \theta^+ + dd^c\psi$ for $\psi \leq 0$.  From invariance and our construction of $g^+$ it follows that $\psi \leq g^+$.  As Forn{\ae}ss and Sibony \cite{FoSi} have observed, this implies that $T^+$ is extremal among $f^*$-invariant currents, which is a form of ergodicity.
\end{rem}

\subsection{Lelong numbers of $T^+$}

It is important for us have a good  control on singularities, i.e. Lelong numbers, of $T^+$.  
The first proposition gives some information about how Lelong numbers of a positive closed current transform under pullback.

\begin{prop}[Theorem 2 and Proposition 5 in \cite{Fav}]
\label{pullback lelong 1}
Let $T$ be a positive closed $(1,1)$ current on $X$.  Then there is a constant $C>0$ such that $p\in X\setminus I_f$ implies that 
\begin{equation}\label{eq:favre}
 \nu(T,f(p)) \leq \nu(f^*T,p) \leq C\nu(T,f(p)). 
\end{equation}
If also $p\notin \cE_f$, then $C\leq \lambda_2(f)$ may be taken to be the local topological degree of $f$ at $p$.  
\end{prop}

The argument for the following result is due to Favre \cite{Fav2}.  We include it for convenience.

\begin{thm}
Assume that $f$ is $1$-stable and has small topological degree. Suppose $p\in X$ is such that $f^n(p)\notin I_f$ for every $n\in\N$.  Then the Lelong number $\nu(T^+,p)$ of $T^+$ vanishes at $p$.  In particular $T^+$ does not charge curves.
\end{thm}

\begin{proof}
Suppose additionally that $f^n(p) \notin\cE_f$ for any $n \in \N$.  Then Proposition \ref{pullback lelong 1} gives 
$$
\nu(T^+,p) = \frac{1}{\lambda_1^n}\nu(f^{n*} T^+,p) \leq  \left( \frac{\lambda_2}{\lambda_1} \right)^n \nu(T^+,f^n(p)).
$$ 
The Lelong numbers of $T^+$ are moreover uniformly bounded above by a constant depending only the cohomology class $\alpha^+$.  Since $\lambda_1>\lambda_2$, we conclude that $\nu(T^+,p)=0$.  Indeed the weaker upper bound in \eqref{eq:favre} implies the same even if $f^n(p) \in \cE_f$ for finitely many $n\in\N$.

On the other hand $f^n(p)\in \cE_f\setminus I_f$ implies that $f^{n+1}(p)$ lies in the finite set $I_f^-$.  So if
$f^n(p)\in\cE_f$ for infinitely many $n$, it follows that $p$ is preperiodic.  Since $\gamma^+$ is finite away from $I_f$, it follows that $g^+$ is finite at $p$.  So $\nu(T^+,p)=0$.
\end{proof}

The pullback $f^*T$ of a positive closed $(1,1)$ current $T$ tends to have non-zero Lelong numbers at points in $I_f$, even if $T$ itself is smooth.  In order to strengthen the convergence in Theorem \ref{thm:cv+}, we need a precise version of this assertion.

\begin{prop} 
\label{pullback lelong 2}
There is a constant $c>0$ such that for any positive closed $(1,1)$ current $T$ that does not charge $\cE^-_f$ and any
$p\in I_f$,
$$
c^{-1}\pair{T}{f(p)} \leq \nu(f^*T,p) \leq c \pair{T}{f(p)}.
$$
\end{prop}

\begin{proof} 
Throughout the proof we will use $\simeq$ to denote equality up to a positive multiple that depends only on $f$.

Fixing $p\in I_f$, we factor the projection $\pi_1:\Gamma\to X$ from the graph of $f$ onto its domain as $\pi_1 = \tilde\pi_1\circ\sigma$ where $\sigma$ is an ordinary point blowup with exceptional curve $E_\sigma\subset \pi_1^{-1}(p)\subset \Gamma$.  Since $\Gamma$ is the \emph{minimal} desingularization of the graph of $f$, it follows that $E_\sigma\not\subset E_{\pi_2}$.  Otherwise we could replace $\Gamma$ with $\sigma(\Gamma)$, $\pi_1$ with $\tilde\pi_1$ and $\pi_2$ with $\pi_2\circ\sigma^{-1}$ and obtain a `smaller' desingularization of the graph.  Hence $\pi_2^* T$ does not charge $E_\sigma$.  

Applying Proposition \ref{pushpull2} to $\pi_1$ and $f^*T$ tells us that
$$
\pi_1^* f^*T = \pi_1^*\pi_{1*}\pi_2^* T = \pi_2^* T + E(T),
$$
where $E(T)$ is an effective divisor supported on $\cE_{\pi_1}$ and depending linearly on the intersection numbers 
$\pair{\pi_2^* T}{C}$, with $C\subset \cE_{\pi_1}$ irreducible.  In addition, because $T$ is positive and does not charge $f(p)$, it follows that $\pair{\pi_2^* T}{C} = \pair{T}{\pi_{2*}C}\geq 0$ for all $C\subset \pi_1^{-1}(p)$.  Therefore we may apply the last assertion in the Proposition \ref{pushpull2} together with the fact that $\pi_2^* T$ does not charge $E_\sigma$ to obtain 
$$
\pi_1^*f^*T|_{E_\sigma} = E(T)|_{E_\sigma} \simeq \pair{\pi_2^*T}{\pi_1^{-1}(p)}E_\sigma = \pair{T}{\pi_{2*}\pi_1^{-1}(p)} E_\sigma
\simeq \pair{T}{f(p)} E_\sigma
$$
So taking a generic point $q\in E_\sigma$, we have
$$
\pair{T}{f(p)} \simeq \nu(\pi_1^* f^* T,q) \simeq \nu(f^*T,p).
$$ 
The righthand equivalence comes from applying Proposition \ref{pullback lelong 1} with $\pi_1$ in place of $f$.
\end{proof}

\begin{defn}\label{def:spurious}
An indeterminacy point $p\in I_f$ is \emph{spurious} if $\pair{\alpha^+}{f(p)} = 0$.
\end{defn}

The possibility of spurious indeterminacy points is a  source of technical difficulties in the sequel (in particular Theorem \ref{thm:cvkahlerpullb} and also \cite{part2}). If  $\lambda_2=1$, we can always remove spurious indeterminacy points, without affecting 1-stability, by performing a modification $X\to \check X$ (see \cite[Proposition 4.1]{BD}). Notice also that if $\alpha^+$ is K\"ahler, there are no spurious indeterminacy points.

It will be useful later to have the following consequence of the previous two results.

\begin{prop} 
\label{small lelong}
Suppose that $f$ is 1-stable and has small topological degree.  Then
given $\varepsilon>0$, there exists an integer $N\in\N$ such that for any positive closed $(1,1)$ form $\omega$, any 
$n\geq N$ and any $p\in X$, we have 
$$
\nu(\lambda_1^{-n}f^{n*}\omega,p) < \varepsilon\norm[H^{1,1}]{\omega}
$$ 
unless $f^j(p)$ is a non-spurious point in $I_f$ for some $j\leq N$.
\end{prop}

\begin{proof}
Fix $p\in X$ and $n\in\N$.  If $p\notin I_{f^{n-1}}$, then it is immediate from Proposition \ref{pullback lelong 1} that $\nu(f^{n*}\omega,p) = 0$.  Otherwise, there is a smallest $k\in\{0,\dots,n-1\}$ such that $f^k p \in I_f$.  Since $f$ is 1-stable it follows that $f^j p \notin \cE_f$ for any $j<k$. Hence Propositions \ref{pullback lelong 1}
and \ref{pullback lelong 2} give 
$$
\nu(f^{n*}\omega,p) \leq \lambda_2^k \nu(f^{(n-k)*}\omega,f^k(p)) \leq \lambda_2^k\pair{f^{(n-k-1)*}\omega}{f^{k+1}(p)}
\leq C\lambda_2^k \lambda_1^{n-k} \norm[H^{1,1}]{\omega}
$$
where $C$ is a constant that does not depend on $p$, $n$ or $\omega$.  Dividing by $\lambda_1^n$ shows that if $p\notin I_{f^N}$ for $N\in\N$ large enough, then $\nu(\lambda_1^{-n}f^{n*}\omega,p) <\varepsilon$.

If  $f^k(p)\in I_f$ is spurious, then we may write the cohomology class of $\omega$ as $c\alpha^+ + \beta$
where $c\geq 0$, $\pair{\alpha^-}{\beta} = 0$, and $c,\norm[H^{1,1}]{\beta}\leq c'\norm[H^{1,1}]{\omega}$.
So from Theorem \ref{spectral}, we find that
$$
\nu(f^{n*}\omega,p)\leq \lambda_2^k\pair{f^{(n-k-1)*}\omega}{f^{k+1}(p)} = \lambda_2^k\pair{f^{(n-k-1)*}\beta}{f^{k+1}(p)}
                   \leq c''\lambda_2^{k+(n-k-1)/2} \norm[H^{1,1}]{\omega}.
$$
Dividing by $\lambda_1^n$ and taking $n\geq N$ large enough gives again that $\nu(\lambda_1^{-n}f^{n*}\omega,p) <\varepsilon$.
\end{proof}

\subsection{Pullbacks of K\"ahler forms}

We study here the convergence of normalized pull-backs of arbitrary closed (1,1) forms.  If the class $\alpha^+$ is K\"ahler (or more generally if there are no spurious indeterminacy points),
 the following result is much easier to prove.  At this level of generality, however, it is new. 
 Our argument depends in particular on the information about Lelong numbers in Proposition \ref{small lelong} and on some volume estimates 
from \cite{G2}.  

\begin{thm}\label{thm:cvkahlerpullb} Assume that $f$ is 1-stable with small topological degree.
Let $\omega$ be any K\"ahler form on $X$. Then
$$
\lim_{n\to\infty} \lambda_1^{-n} f^{n*} \omega = \pair{\omega}{\alpha^-} T^+,
$$
\end{thm}

We remark that the conclusion of Theorem \ref{thm:cvkahlerpullb} applies more generally to differences of K\"ahler forms and hence to any smooth closed real (1,1) form.  

\medskip
\noindent{\it Proof.}  We assume with no loss of generality that $\pair{\omega}{\alpha^-} = 1$, so that the cohomology class of $\lambda_1^{-n} f^{n*} \omega$ tends to that of $T^+$.  We recall the notation $\theta^+, g^+$ from the proof of Theorem \ref{thm:cv+}.  For each $n\in\N$ we write
$$
T_n = \theta^+ + \eta_n + dd^c w_n,
$$
where $w_n\in L^1(X)$ is normalized so that $\sup_X w_n = 0$, and $\eta_n$ is a smooth closed $(1,1)$ form with $\pair{\eta_n}{\alpha^-} = 0$. Theorem \ref{spectral}  imply that $\norm[H^{1,1}]{\eta_n} \to 0$ as $n\to\infty$, so 
we may assume that $-c_n\omega \leq \eta_n \leq c_n\omega$, where $c_n>0$ decreases to zero as $n\to\infty$.   Since the $w_n$ are $\theta^+ + c_0\omega$-plurisubharmonic and normalized, we see that $(w_n)_{n\in\N}$ is relatively compact in $L^1(X)$.  

Now we introduce a second index $k\in\N$ and estimate
$$
T_{n+k} \leq \frac{1}{\lambda_1^k}(f^{k*}\theta^+ + c_n f^{k*}\omega + dd^c w_n\circ f^k).
$$
Since $\norm[H^{1,1}]{\lambda^{-k}_1 f^{k*}\omega^+}\leq C$ uniformly in $k$, we can replace $c_n$ by $Cc_n$ to get 
$$
T_{n+k} \leq (1+c_n) \theta^+ + c_nc_k \omega + dd^c(g_k^+ + c_n w_k + \lambda_1^{-k} w_n\circ f^k).
$$
Setting $u_{n,k} = g_k^+ + c_n w_k + \lambda_1^{-k}w_n\circ f^k$, we claim that $\{u_{n,k}:n,k\in\N\}$ is a relatively compact family of functions.  Each $u_{n,k}$ is negative and $\theta^+ + c_0\omega$-plurisubharmonic, so it suffices to show that there is no sequence $(u_{n_j,k_j})_{j\in\N}$ tending uniformly to $-\infty$ on $X$.  

We will do this by finding $M\in\R$ such that 
$
\Vol\{u_{n,k}<-M\} < \Vol(X)
$
for all $n,k\in\N$.  We have already seen that $g_k^+ \to g^+$ and that $(w_k)_{k\in\N}$ is relatively compact in $L^1$, so $\Vol\{g_k^+<-M_1\}<\Vol(X)/3$ and $\Vol\{w_k<-M_2\} < \Vol(X)/3$ and for some $M_1,M_2\in\R$.  Setting $w_{n,k} \eqdef \lambda_1^{-k}w_n\circ f^k$, and taking $M_3 = t$ large enough in  next lemma, we find that $\Vol\{w_{n,k} < -M_3\}<\Vol(X)/3$ for all $n,k\in\N$.  Thus $M = M_1 + c_0M_2 + M_3$ suits our need.

\begin{lem}
\label{volume bound}
There exist constants $\kappa,\tau$ such that for any $t>0$
$$
\Vol\{p\in X: w_{n,k}(p)\leq -t\} \leq \frac{\kappa}{t - \tau\lambda_1^{-k}}.
$$
\end{lem}

\begin{proof}
Since the non-positive $c\omega$-plurisubharmonic functions $w_n$ form a relatively compact sequence, it follows (see e.g. \cite{Z}) that there are constants $A,B\geq 0$ such that $\int e^{-Aw_n} \,\omega_X^2 < B$ for all $n\in\N$.  Hence, 
$$
\Vol\{w_n \leq -t\} \leq Be^{-At}.
$$
Thus, if $\Omega_{n,k}(t) = \{p\in X:w_{n,k} \leq -t\}$ we have from \cite{G2} that
$$
Be^{-A\lambda_1^k} \geq \Vol\{w_n \leq -t\lambda_1^k\} = \Vol f^k\Omega_{k,n}(t) \geq \exp(-D\lambda_1^k/\Vol\Omega_{n,k}(t)),
$$
where the constant $D$ depends only on $f$).  Rearranging completes the proof.
\end{proof}

Note that the above discussion implies that the family $\{w_{n,k}:n,k\in\N\}$ is relatively compact in 
$L^1(X)$; i.e. $w_{n,k} = u_{n,k} - (g_k + w_k)$ is a difference of functions from relatively compact families.

Suppose now that $T = \lim_{j\to\infty} T_{m_j}$ is a limit point of the sequence of interest.  We will complete the proof of Theorem \ref{thm:cvkahlerpullb} by showing that $T\leq T^+$.  Refining the given subsequence, we may assume that 
$m_j = n_j + k_j$, where $(n_j)$ and $(k_j)$ increase to infinity as quickly as we please.  By compactness, we may also assume that $w_{n_j,k_j} \to W\in L^1(X)$.  Thus
$$
\frac{f^{m_j*}\omega}{\lambda_1^{m_j}} = \frac{f^{(n_j+k_j)*}\omega }{\lambda_1^{n_j+k_j}} 
	\leq (\theta^+ + dd^c g^+_{k_j}) + c_{n_j} \omega + dd^c(c_{n_j} w_{k_j} + w_{n_j,k_j})
        \to T^+ + dd^c W,
$$
since $(w_k)$ is relatively compact and $c_n\to 0$ as $n\to\infty$.  Since by our normalization, $\lambda_1^{-m_j} f^{m_j*}\omega$ converges to $T^+$ in cohomology, the proof of Theorem \ref{thm:cvkahlerpullb}
is therefore concluded by 

\begin{lem}
 If $n_j, k_j$ are chosen appropriately, then for every $t>0$
$$
\lim_{j\to\infty}\Vol\{w_{n_j,k_j}\leq -t\} = 0.
$$
\end{lem}

\begin{proof}
Fix $j\in\N$.  By Proposition \ref{small lelong}, there exists $n_j\in\N$ such that $\nu(\lambda^{-n_j}f^{n_j*}\omega,p) < 1/j$ unless $f^{\ell}(p)$ is a non-spurious point in $I_f$ for some $\ell<n_j$.  Let us denote the finite set of exceptional $p$ by $I_j$.  For $p\in I_j$, we claim that $\nu(g^+,p) \eqdef \nu(T^+,p) >0$.  Indeed, if $f^\ell(p)$ is a non-spurious point of indeterminacy, then Propositions \ref{pullback lelong 1} and \ref{pullback lelong 2} give us
$$
\nu(T^+,p) = \lambda_1^{-\ell} \nu(f^{\ell*} T^+,p) \geq \nu(T^+,f^\ell(p)) \geq C\pair{T^+}{f^{\ell+1}(p)} > 0
$$

Now let $\chi:X\to[0,1]$ be a smooth function equal to $1$ near $I_j$ and vanishing in a neighborhood of each point in $I_{f^{n_j}-1} - I_j$.  Then $0\geq \chi w_{n_j} \geq c g^+$ for some $c = c(n_j) \geq 0$.  In particular, 
$(\chi w_{n_j})\circ f^k \to 0$ in $L^1(X)$ as $k\to\infty$.  So for $k_j$ large enough, we can assume that $\Vol\{(\chi w_{n_j})\circ f^{k_j}\leq -j^{-1/2}\} \leq j^{-1/2}$.

On the other hand, the Lelong numbers at singularities of $(1-\chi) w_{n_j}$ are all less than $1/j$.  Hence we may refine the initial volume estimate in Lemma \ref{volume bound} to read
$$
\Vol\{(1-\chi)w_{n_j} \leq -t\} \leq Be^{-tj}
$$
for a constant $B$ depending on   $n_j$ \cite{K}.
Proceeding as before, we arrive at the estimate
$$
\Vol\{\lambda_1^{-k_j}((1-\chi)w_{n_j})\circ f^{k_j} \leq -t\} \leq \frac{\kappa}{jt-\tau_n\lambda_1^{-k_j}}.
$$
with $\kappa$ independent of $j$.  Taking $t=j^{-1/2}$ and suitably increasing $k_j$, we have
$$
\Vol\{\lambda_1^{-k_j}((1-\chi)w_{n_j})\circ f^{k_j} \leq -j^{-1/2}\} \leq 2\kappa j^{-1/2}.
$$
We now put our estimates together to find
$$
\lim_{j\to\infty} \Vol\{w_{n_j,k_j} \leq -j^{-1/2}\} \leq Cj^{-1/2}
$$
for some $C$ and all $j\in\N$.  Letting $j\to\infty$ completes the proof.
\end{proof}

The following consequence will be useful for proving the laminarity of $T^+$. 

\begin{cor}\label{cor:cv}
Assume, under the hypotheses of Theorem \ref{thm:cvkahlerpullb}, that $X$ is projective with fixed embedding $X\hookrightarrow \cp^N$.  Then for Lebesgue a.e. hyperplane section $L$, we have
$$
\frac{1}{\lambda_1^n}(f^n)^* [L] \rightarrow c T^+,
$$
where $c$ depends only on the embedding. 
\end{cor}

\begin{proof} We have the Crofton formula for the Fubini-Study form $\om_{\rm FS}$ on 
$\cp^N$: $\om_{\rm FS}=\int_{\check{\mathbb{P}}^N}  [H]\, dv(H)$, where $dv$ denotes Fubini-Study
volume on the dual of $\mathbb{P}^n$ (see \cite{Cir}).  So if $\om_{\rm FS}\rest{X}$ denotes the K\"ahler
form induced by $\om_{FS}$ on $X$, we can restrict to get $\om_{\rm FS}\rest{X}=\int [H \cap X]\, dv$.

For each hyperplane $H$, we have $[H] - \om_{\rm FS} =dd^c\psi_H$, where $\psi_H(p) = \psi(H,p) \leq 0$.  Thus 
$[H\cap X] - \om_{\rm FS}\rest{X} = dd^c(\psi_H|_X)$ as long as $X$ is not contained in $H$.  Since $\lambda_1^{-n}(f^n)^*(\om_{\rm FS}\rest{X})\rightarrow T^+$, it is enough to prove for a.e. $H$ that 
$\lambda_1^{-n} \psi_H\circ f^n\rightarrow 0$ in $L^1(X)$. This follows from Fubini's Theorem and the fact that 
$\int \psi_H \,dv$ is independent of $H$.  The reader is referred to \cite[Theorem 1.10.1]{Sib} for  more results in this direction.
\end{proof}

\subsection{Laminarity of $T^+$}\label{subs:lamin}

The geometric structure of the invariant currents will play a central role \cite{part2, part3} in the fine study of the 
ergodic properties of our mappings.  Recall that a positive (1,1) current is {\em laminar} if it can be written as an integral of a family of holomorphic disks in which members have no isolated intersections.  A current $T$ is {\em uniformly laminar} if the disks form a lamination of some open subset of $X$, and $T$ is a foliation cycle associated to this lamination.  One can show that any laminar current is an increasing limit of uniformly laminar currents. A laminar current on $X$ is {\em strongly approximable} if it is a limit of compact subvarieties with controlled geometry (see \cite{Du1} for precise details). Strong approximability implies a certain quantitative estimate on the `rate' of approximation by uniformly laminar currents, and this will be important for \cite{part2, part3}. 
We refer the reader to \cite{BLS, Du1, Du5, Du4} for more details about  laminarity and its consequences.

\begin{thm}\label{thm:lamin}
Assume $X$ is projective and $f$ is $1$-stable with small topological degree $\lambda_2<\lambda_1$. Then $T^+$ is a strongly approximable laminar current. 
\end{thm}

\begin{proof}
The theorem was proved in \cite[Prop. 4.2]{Du1} under two additional assumptions. First, the surface $X$ was supposed to be
rational, which was only a matter of convenience: it is enough (see \cite{Ca}) to replace pencils of lines by pencils of hyperplane sections everywhere, and to project along these pencils to treat the case of general $X$.  More seriously,
there was an extra hypothesis (H) on the relative positions of the total indeterminacy set $I(f^\infty)$ and the singularities of the graph of $f:X\rightarrow X$.  Here, following \cite[Theorem 1]{Du1} closely,
we explain how to remove this assumption.

\medskip

We have to prove the following: let $L$ be a generic hyperplane section of $X$, and $C_n=f^{-n}(L)$, then 
\begin{equation}\label{eq:genus}
\mathrm{genus}(C_n)+ \sum_{x\in \mathrm{Sing}(C_n)} n_x(C_n) =
O(\lambda_1^n).
\end{equation} 
Here genus means geometric genus, and $n_x(C_n)$ is the number of local irreducible components of $C_n$ at $x$. The two terms on the left side are estimated separately and by induction in \cite[Lemmas 4.3 and 4.4]{Du1}.  We show how to adapt the proofs to the general case. 

We let $\Gamma$ be the desingularized graph of $f$, endowed with the two
natural projections $\pi_i:\Gamma\rightarrow X$. By induction, we define $\Gamma^n$ to
be the minimal smooth surface such that all descending arrows in the following diagram are holomorphic:
\begin{center}~
\xymatrix@dr@+1pc{\Gamma^n \ar[r]^{\omega} \ar[d]_{\eta} 
               \ar@/^2pc/[rr]^{\pi_{2,n}} 
     \ar@/_2pc/[dd]_{\pi_{1,n}}
    &\Gamma^{n-1} \ar[r]|{\pi_{2,n-1}}
         \ar[d]|{\pi_{1,n-1}} & X \\
          \Gamma \ar[r]^{\pi_2} \ar[d]_{\pi_1} & X \ar@{-->}[ur]_{f^{n-1}}\\
          X \ar@{-->}[ur]_{f}}
\end{center}
Notice that $\pi_1$, $\eta$ and $\pi_{1,n}$ are compositions of point blowups, while the topological degree of $\pi_{2,n}$ is
$\lambda_2^n$.  Recall that $\cE_h$ denotes the exceptional locus of a non-degenerate holomorphic map $h:Y\to Z$ between surfaces and that we regard $\cE_h$ alternately as a {\em reduced} effective divisor.

We choose the hyperplane section $L$ according to the genericity assumptions (G1) and (G3) in \cite{Du1}.  That is, first of all we apply Corollary \ref{cor:cv} to choose $L$ so that $\lambda_1^{-n}[(f^n)^*(L)]$ converges to a positive multiple of $T^+$.
Secondly, we take $L$ to miss the finite set $\pi_{2,n}(\cE_{\pi_{2,n}})$ for every $n\in\N$.
Bertini's theorem moreover allows us to assume that $\widehat{C}_n\eqdef \pi_{2,n}^* L$ is smooth, reduced and irreducible.

As $\pi_{1,n}$ is a composition of point blow-ups, $\pi_{1,n}:~ \widehat{C}_n \rightarrow C_n$ is a resolution of
singularities. Hence 
$$
\sum_{x\in Sing(C_n)} n_x(C_n) \leq \# \widehat{C}_n \cap \mathcal{E}_{\pi_{1,n}}
   \leq \pair{\widehat{C}_n}{\cE_{\pi_{1,n}}}
$$ 
and the geometric genus of $C_n$ is the usual genus of $\widehat{C}_n$. 

\medskip

In \cite[Lemma 4.3]{Du1}, the assumption (H) is invoked only to prove the estimate 
$$
\pair{\eta^* \cE_{\pi_1}}{\widehat{C}_n} = \pair{\cE_{\pi_1}}{\eta(\widehat{C}_n)}\leq C^{st}\lambda_1^n.
$$
To get rid of this dependence, we observe that $\eta(\widehat{C}_n)$ is an irreducible curve which 
projects to $C_n$ by $\pi_1$ and to $C_{n-1}$ by $\pi_2$. In particular $\eta(\widehat{C}_n)$ is the
proper transform of $C_{n-1}$ under $\pi_2$, so 
$$
\pi_2^{-1} (C_{n-1}) = \eta(\widehat{C}_n) + D_n,
$$ 
where $D_n$ is an effective divisor supported on $\cE_{\pi_2}$.  We claim that $D_n \leq C^{st} \lambda_1^n \cE_{\pi_2}$. Granting this momentarily, we deduce
\begin{align*}
\pair{\cE_{\pi_1}}{\eta(\widehat{C}_n)} &=
 \pair{\cE_{\pi_1}}{\pi_2^*(C_{n-1})}-\pair{\cE_{\pi_1}}{D_n} \\
&\leq C^{st} \left( \norm[H^{1,1}]{C_{n-1}} 
+ \lambda_1^n 
\right) \leq C^{st}\lambda_1^n,
\end{align*}
by Theorem \ref{spectral}.

To prove our claim about $D_n$, we recall that the multiplicity $\mathrm{mult}_{p} (C_{n-1})$ of $C_{n-1}$ at any
point $p \in \pi_2(\mathcal{E}(\pi_2))$ is the number of intersection points near $p$ of $C_{n-1}$ with a generic hyperplane section $S$. Thus
$$
\mathrm{mult}_{p} (C_{n-1}) \leq \pair{C_{n-1}}{S} \leq C^{st}\norm[H^{1,1}]{C_{n-1}} = O(\lambda_1^n).
$$
Furthermore, the multiplicity of an irreducible component $V\subset\cE_{\pi_2}$ in $D_n$ is just the Lelong number of 
$D_n$ at a generic point in $V$.  Hence Favre's estimate on Lelong numbers (Proposition \ref{pullback lelong 1}) tells us that
this multiplicity is bounded above by $C^{st}\mathrm{mult}_{\pi_2(V)} C_{n-1}$.  This proves the claim.

The argument for adapting \cite[Lemma 4.4]{Du1} is similar. The assumption $(H)$ is used to prove that 
$\pair{\eta^*R_{\pi_2}}{\widehat{C_n}} = O(\lambda_1^n)$, where $R_{\pi_2}$ is the
ramification divisor of $\pi_2$. As before, $\pair{\eta^*R_{\pi_2}}{\widehat{C_n}} =
\pair{R_{\pi_2}}{\eta(\widehat{C_n}}$, with $\eta(\widehat{C}_n)= \pi_2^{-1} (C_n) - D_n$, where $D_n$ is an
effective divisor supported on $\cE_{\pi_2}$.  The desired control then follows from a cohomological computation similar to the one above.
\end{proof}

\section{The canonical current $T^-$}\label{sec:T-}
\label{subs:cv-}

If $f$ is bimeromorphic, then by applying Theorem \ref{thm:cv+} to $f^{-1}$ one obtains an $f_*$ invariant current $T^-$ with properties analogous to $T^+$.  We show in this section that $T^-$ exists under the weaker hypothesis that $f$ has small topological degree.  Thus we assume throughout that the meromorphic map $f$ is $1$-stable with 
$\lambda_2<\lambda_1$.

\subsection{Construction of $T^-$}

\begin{thm}\label{thm:cv-}
Let $\theta^-$ be a smooth closed (1,1)-form, or more generally a closed $(1,1)$ current with bounded potentials,  representing the class $\alpha^-$.  Then the sequence $\lambda_1^{-n} f^n_* \theta^-$ converges weakly to
a positive closed (1,1)-current $T^- = \lambda_1^{-1}f_*T^-$ that is independent of $\theta^-$. 
\end{thm}

This theorem has been already observed   in some special (non-invertible) cases, e.g. \cite[Theorem 5.1]{G4}.  In this generality it is new.  We begin with a technical observation about pushing forward.

\begin{lem}
\label{lem:pushcts}
Let $g:X\to Y$ be a dominating meromorphic map between compact complex surfaces. Let $U\subset Y\setminus I_g^-$ and $W\supset g^{-1}(U)$ be open sets. If $\psi\in L^1(X)$ is continuous on $W$, then $g_*\psi$ is continuous on $U$.  
Similarly, if $\omega$ is a closed $(1,1)$ current with continuous local potentials on $W$, then $g_*\omega$ has continuous local potentials on $U$.
\end{lem}

\begin{proof}
Let $\Gamma$ be the desingularized graph of $g$ and $\pi_1,\pi_2$ the projections onto $X$ and $Y$.  Then $\pi_1^*\psi$ is continuous on $\pi_1^{-1}(W)$.  Since $I_g\cap U= \emptyset$, we may shrink $W$ so that $W\cap \cE_g =\emptyset$.  Hence for any $p\in U$, we obtain that $\pi_1^*\psi$ is constant on each connected component $C$ of $\pi_2^{-1}(p)$.  It follows then that
$$
g_*\psi(p) = (\pi_{2*}\pi_1^*\psi)(p) = \sum_{C\subset \pi_1^{-1}(p)} (\pi_1^*\psi)(C)
$$
is continuous at $p$.  

To get the corresponding result for $\omega$, it suffices to fix $p\in Y\setminus I_g$ and let $U$ be a neighborhood of $p$.  In particular, $g^{-1}(p)$ is finite, so by choosing $U$ and then $V$ small enough, we may assume that $\omega|_V = dd^c v$ for some continuous function $v$ on $X$.  Now we get what we need from
$$
g_*\omega|_{U} = (g_*dd^c v)|_U = (dd^c g_* v)|_U.
$$
\end{proof}

The rest of the proof of Theorem \ref{thm:cv-} is similar to that of Theorem \ref{thm:cv+}, so we only sketch it. 
Let $\theta^-$ be a smooth representative of $\alpha^-$.  We can write 
$$
\frac{1}{\lambda_1}f_* \theta^-=\theta^-+dd^c \gamma^-, 
$$
with $\gamma^-\in L^1(X)$.  Lemma \ref{lem:pushcts} implies that $\gamma^-$ is continuous away from $I_f^-$. 
Since the class $\alpha^-$ is represented by a positive current with bounded potentials,  $\gamma^- $ is bounded from above, and it is no loss of generality to assume that $\gamma^-\leq 0$.
 The sequence $g_n^- \eqdef \sum_{j=0}^{n-1} \lambda_1^{-j}f^j_*\gamma^-$ is therefore  decreasing. 
To conclude that the sequence $g_n^-$ converges, we need to prove that it is bounded from below by a $L^1$ function. 
For this, as in Theorem \ref{thm:cv+}, we apply Sibony's argument from \cite{Sib}. This is where we use the assumption of small topological degree,  
for it implies for any constant $C>0$ that $f_* C = \lambda_2 C < \lambda_1 C$.  So if $u:X\to\R$ is
bounded above by $C$, then $\lambda_1^{-n} f^n_* u$ is  bounded by $\frac{\lambda_2^n}{\lambda_1^n} C < C$. 
See \cite[Theorem 5.1]{G4} for more details.

Let $g^- = \lim g_n^-$, and  $T^- = \theta^- + dd^c g^-$.   The positivity of $T^-$, its invariance 
and independence from $\theta^-$ are shown as in Theorem \ref{thm:cv+}.
\qed

One can argue as in Remark \ref{minsing} that $T^-$ has minimal singularities among invariant currents.  It is unclear to us, however, how bad these singularities might be.  For instance, we do not know whether there exists a map $f$ for which $T^-$ charges a curve.

\subsection{Convergence towards $T^-$}

\begin{thm} \label{thm:cvkahlerpushf}
Let $\omega$ be a K\"ahler form on $X$. Then
$$
\frac{1}{\lambda_1^n} f^n_* \omega \longrightarrow c T^-,
\text{ where } c=\pair{ \omega } {\alpha^+}=\int_X \omega \wedge T^+.
$$
\end{thm}

Although this result is analogous to Theorem \ref{thm:cvkahlerpullb}, we certainly cannot use
the same proof, for we do not have volume estimates for pushforwards.  We work instead by duality, using a stronger version of Theorem \ref{thm:cvkahlerpullb} that applies to all smooth real, not necessarily  closed, (1,1) forms.

\begin{lem} \label{lem:cvtestpb}
Let $\chi$ be a smooth test function on $X$. Then
$$
\frac{1}{\lambda_1^n}(f^n)^*(\chi \omega) \longrightarrow c T^+,
\text{ where } c=\int \chi \omega \wedge T^-.
$$
\end{lem}

\begin{proof}
We follow the now standard approach from \cite{Sib} (see also \cite{bs3, fs98}).  Let $\mathcal{S}$ denote the  set of cluster points of the (relatively compact) sequence $S_n:=\lambda_1^{-n}(f^n)^*(\chi \omega)\geq 0$.  We can assume without loss of generality that $0 \leq \chi \leq 1$ and $\int \omega \wedge T^-=1$.  Then from Theorem \ref{thm:cvkahlerpullb}, it follows for any $T\in\mathcal{S}$ that $0\leq T\leq cT^+$.

We next argue that elements $T\in \cS$ are closed.  Since $S_n$ is real, it suffices to estimate the mass $\mass{\partial S_n}$ of $\partial S_n$.  Fixing a smooth $(0,1)$ form $\alpha$, we use the Cauchy-Schwarz inequality to estimate
\begin{eqnarray*}
\left| \langle \partial S_n, \overline{\alpha} \rangle \right|
&\leq &\frac{1}{\lambda_1^n}
\langle \partial \chi \circ f^n \wedge \overline{\partial} \chi \circ f^n, f^{n*}\omega \rangle^{1/2}
\langle \alpha \wedge \overline{\alpha}, f^{n*}\omega \rangle^{1/2}  \\
&=& \frac{\lambda_2^{n/2}}{\lambda_1^n} 
\langle \partial \chi \wedge \overline{\partial} \chi , \omega \rangle^{1/2}
\langle \alpha \wedge \overline{\alpha}, f^{n*}\omega \rangle^{1/2} 
\leq ||\alpha|| \left( \frac{\lambda_2}{\lambda_1} \right)^{n/2} 
\langle \partial \chi \wedge \overline{\partial} \chi , \omega \rangle^{1/2}.
\end{eqnarray*}
Thus $\mass{\partial S_n}=O((\lambda_2/\lambda_1)^{n/2}) \rightarrow 0$.

Now if $T = \lim_{j\to\infty} S_{n_j} \in \mathcal{S}$ is the limit of some subsequence, then after refining the subsequence, we may also assume that $S =  \lim_{j\to\infty} S_{n_j+1}\in \mathcal{S}$ exists.  Since $T,S\leq cT^+$ do not charge the critical set of $f$, we have that $\lambda_1^{-1}f^*T = S$.  Similarly, we can refine to arrange that $S = \lim_{j\to\infty} S_{n_j-1}\in \mathcal{S}$ exists, and then $T= \lambda_1^{-1}f^*S$.  We infer that $f^*\mathcal{S} = \lambda_1^{-1}\mathcal{S}$.

Finally, for each $T\in\mathcal{S}$, we write $T=\theta^+ + dd^c u_T$, $cT^+ - T = (c-1) \theta^+ + dd^c v_T$ where by hypothesis both $u_T$ and $v_T$ are qpsh, and we normalize so that $\int u_T \omega_X^2 = \int v_T \omega_X^2 = 0$.  Since $\mathcal{S}$ is compact we have $M\geq 0$ such that $u_T,v_T \leq M$ for all $T\in\mathcal{S}$.  So if $\tilde g^+$ is the quasipotential for $T^+$ with $\omega_X^2$ mean zero, we obtain that
$$
M\geq u_T = c\tilde g^+ - v_T \geq c\tilde g^+ - M'.
$$
As $\lambda_1^{-n}\tilde g^+\circ f^n \to 0$ in $L^1(X)$, we infer that $\lambda_1^{-n}u_T\circ f^n \to 0$ uniformly in $T$.  That is,
$$
\lambda_1^{-n}f^{n*} T = \lambda_1^{-n}f^{n*}\theta^+ + \lambda_1^{-n} dd^c u_T\circ f^n \to T^+ 
$$
uniformly on $\mathcal{S}$.  Together with complete invariance of $\mathcal{S}$, this implies $\mathcal{S}=\{T^+\}$.
\end{proof}

\medskip
\noindent{\it Proof of Theorem \ref{thm:cvkahlerpushf}.}
It suffices to prove convergence on a generating family of test forms,
e.g. forms of the type $\theta=\chi \omega'$, where $\chi$ is a test function and 
$\omega'$ is a K\"ahler. By Lemma \ref{lem:cvtestpb}
$$
\left \langle \frac{1}{\lambda_1^n} f^n_* \omega, \theta \right \rangle
=\left \langle  \omega, \frac{1}{\lambda_1^n} (f^n)^*\theta \right \rangle
\longrightarrow \langle \omega, c' T^+ \rangle
=\langle cT^-,\theta \rangle,
$$
where $c=\int \omega \wedge T^+=\pair{ \omega }{ \alpha^+}$, as desired.
\qed\medskip

As with $T^+$, it follows that $T^-$ is well-approximated by divisors.

\begin{cor}\label{cor:cv-}
Assume $X$ is projective.
Given any projective embedding of $X$, if $L$ is a generic hyperplane section, 
$\lambda_1^{-n} (f^n)_*L\rightarrow c T^-$, where $c$ depends only on the embedding.
\end{cor}

\subsection{$T^-$ is woven} \label{subs:woven}

Woven currents were introduced by T.C. Dinh \cite{Dinh}.  They appear in a dynamical context in \cite{DTh1}.
The definitions of {\em uniformly woven} and {\em woven} currents are similar to the laminar case, except 
that there is no restriction on the way that members of the underlying family of disks may intersect each other. 
Accordingly, we define a {\em web} to be an arbitrary union of subvarieties of some given open set.

Heuristically, one should not expect $T^-$ to be a laminar current. That is, as we explore further in \cite{part3}, the disks appearing in the woven structure of $T^-$ should be (pieces of) unstable manifolds corresponding to some invariant measure.  It is well known that for noninvertible mappings, there is no well-defined notion of unstable manifold of a point. Rather, through any point $p$ there is an unstable manifolds for each history of $p$ (i.e. each infinite backward orbit starting at $p$). So in the absence of special circumstances, $\lambda_2>1$ should imply the existence of an infinite ``bouquet'' of unstable manifolds through almost every point.

\begin{thm}\label{thm:woven}
Assume $X$ is projective.
Then  $T^-$ is a  strongly approximable woven current.
\end{thm}

\begin{proof} The result will follow from the following general criterion \cite{Dinh}.
 Let $C_n$ be a sequence of curves on a projective manifold, such that 
${\rm genus}(C_n)=O(\deg (C_n))$, where ${\rm genus}$ denotes the geometric genus. 
Then any cluster value
of the sequence $(\deg(C_n))^{-1}[C_n]$ is a woven current. The proof is just 
rewriting  the criterion  of \cite{Du1} by replacing ``laminar'' by ``woven'' 
everywhere  (see  \cite[Proposition 5.8]{Du6} for more details on this approach, 
and also \cite{Dinh}), and projecting along 
linear pencils. 

\medskip

From Corollary \ref{cor:cv-},  we have that $\lambda_1^{-n}f^n_*L\rightarrow cT^-$ for almost any hyperplane section $L\subset X$.  Hence for a.e. $p\in X$, the convergence holds for a.e. $L\ni p$.  Choose such a generic $p$.

For each $n\geq 0$, among the hyperplane sections 
through $p$, only finitely many of them meet $f^{-n}(f^n(p))\setminus\set{p}$. 
Thus we get that for a generic $L$ through $p$, and every $n\geq 0$, 
$f^n\rest{L}:L\rightarrow f^n(L)$ has generic degree 1. In particular, $f^n\rest{L}$ is a 
resolution of singularities of $f^n(L)$, so the geometric genus of $f^n(L)$ is constant. Also 
$f^n_*L$ is reduced and irreducible, ie. $f^n_*L = f^n(L)$. From this it follows that $T^-$ 
is woven.
\end{proof}

\section{Rational and irrational examples}\label{sec:examples}

\subsection{Maps on rational surfacees}
Self-maps with small topological degree are abundant on rational surfaces.
The `Cremona group' of birational maps of  $\cp^2$ is itself enormous.   One can get non-invertible examples by composition $f=f_1\circ f_2\circ T$ where $f_1$ is a birational map with $\lambda_1(f_1)$ large, $f_2$ is a holomorphic map with $\lambda_2(f_2) > 1$ small, and $T$ is a suitably generic automorphism.  Indeed the
sets $I_f = (f_2\circ T)^{-1}(I_{f_1})$ and $I_f^- = I_{f_1}^-$ are finite, so for $T$ outside a countable union of subvarieties in $\text{Aut}(\cp^2)$, one has $f^n(I_f^-)\cap I_f = \emptyset$ for all $n\in\N$.  Hence $f$ is $1$-stable and 
$$
\lambda_1(f) = \lambda_1(f_1)\lambda_1(f_2) > \lambda_1(f_2)^2  = \lambda_2(f_2) = \lambda_2(f).
$$

Beyond these generic examples, we have some specific maps of particular interest.

\subsubsection{Polynomial maps of $\C^2$}
Any polynomial map $f:\C^2\self$ can, by extension, be regarded as a meromorphic self-map on $\cp^2$.  As the following example shows, some of these can be seen explicitly to be $1$-stable and of small topological degree.

Let $f:\cp^2\self$ be the map given on $\C^2 = \{[x:y:1]\in\cp^2\}$ by $f(x,y) = (y,
 Q(x,y))$, where $Q$ is a degree $d>1$ polynomial such that the coefficient of $y^d$ is non-zero whereas that of $x^d$ vanishes. It is clear that $I_{f}=[1:0:0]$ and $f(L_\infty \setminus I_f ) = [0:1:0]$ which is fixed. Hence (see the remark following Definition \ref{as}) we see that $f$ is $1$-stable.  The pullback of a  non-vertical line $L$ is a curve of degree $d$  so   $\lambda_1(f) = d$. On the other hand it is clear that the topological degree $\lambda_2(f)$ is $d_x$, where $d_x<d$ is the highest power of $x$ appearing in $Q$.  Thus $f$ has small topological degree.
\medskip

In a much deeper fashion, Favre and Jonsson \cite{FJ2} have recently shown that if $\lambda_1(f) > \lambda_2(f)^2$ there always exists a modification $\pi:X\to\cp^2$ with $X$ smooth and $\pi(\cE_\pi) \cap\C^2 = \emptyset$, such that $V_\infty = X\setminus \pi^{-1}(\C^2)$ is mapped by $f^k$ to a single point $p\in V_\infty\setminus I_{f^k}$.  Thus $f^{k+n}(I^-_{f^k})\cap V_\infty = \{p\}$ for every $n\geq 0$, and since $I_{f^k} \subset V_\infty$, it follows that $f^k$ is $1$-stable.  More precisely, one has $(f^{k+n})^* = (f^*)^n (f^k)^*$ for every $n\geq 0$,
so that the image of $(f^k)^*$ is contained in an $f^*$ invariant subspace $H \subset H^{1,1}(X,\R)$ on which one has ordinary $1$-stability $(f^*)^n = (f^n)^*$.  If, moreover, $\alpha^+ = \lambda_1^{-k} (f^k)^*\alpha^+$ is the invariant nef class for $(f^k)^*$, then we have automatically that $\alpha\in H$.  In particular, $f^*\alpha^+/\lambda_1 \in H$ is another $(f^k)^*$ invariant nef class.  By uniqueness, we infer $f^*\alpha^+ = \lambda_1\alpha^+$; i.e. $\alpha^+$ is invariant under pullback by a \emph{single} iterate of $f$.  From here our construction of the $f^*$ invariant current $T^+$ goes through as above, and one sees easily that this is the same as the invariant current constructed for $f^k$.  Assuming now that $f$ has small topological degree, Theorem \ref{thm:cvkahlerpullb} applies to $T^+$ (as least if we substitute $f^k$ for $f$).  Consequently, $T^+$ is laminar.  After applying similar consideration to pushforwards by $f$ and by $f^k$, we end up with a woven current $T^-$ invariant under pushforward by (a single iterate of) $f$.  In summary, thanks to the stability result of \cite{FJ2} all the key results (i.e. the ones needed in the sequels \cite{part2, part3} of this paper) apply to any polynomial mapping $f:\C^2\to\C^2$ once we choose a good compactification of $\C^2$.  

In fact, Favre and Jonsson further show that when a polynomial has small topological degree, one may construct a continuous affine potential $g^+(x,y)$ for the invariant current $T^+ = T^+|_{\C^2}$.  Hence we can always define the probability measure $\mu = T^+\wedge T^-$.  Nevertheless it is difficult in general to control the potential of $T^-$, so studying this measure 
(even showing that it is invariant) is problematic. We will solve this problem with ideas developed in \cite{part2}. 

Adding some assumptions on $f$ can make the situation much easier: if the line at infinity is repelling in some sense, $T^+\wedge T^-$ has compact support in $\C^2$ and its ergodic properties are studied in \cite{G4}. If $f$ is merely proper, then, with notation as in Section \ref{sec:T-}, the function $\gamma^-$ is locally bounded outside the superattracting point $p$, from which we conclude that $T^-$ has locally bounded potential.   Thus $T^+\wedge T^-$  does not charge the indeterminacy set, and can easily be proved to be invariant and mixing. It can also be proved using the techniques of \cite{ Du5, Du4} that the wedge product is geometric, so in this case the reader can directly jump to  \cite{part3} for the finer  dynamical  properties of $\mu$.

\subsubsection{The secant method}

Two term recurrences based on rational  functions also furnish interesting examples.  An entertaining instance of this is the so-called `secant method' applied to find roots of a polynomial $P:\C\to\C$, with $d = \deg P > 1$: one begins with two guesses $x,y\in \C$ at a root of $P$ and seeks to improve these guesses by finding the unique point $(R(x,y),0)$ (specifically, $R(x,y)=\frac{yP(x)-xP(y)}{P(x)-P(y)}$) on the line through $(x,P(x))$ and $(y,P(y))$.  This gives a rational map $f:\C^2\to\C^2$ of the form $(x,y)\mapsto (y,R(x,y))$ which may then be iterated with the hope of converging to $(z,z)$ for some root $z$ of $P$.  Extending $f$ to $\cp^2$ gives a map for which $[0:1:0] \in I_f\cap I_f^-$, implying that $f$ is not 1-stable.  The extension to a meromorphic map $f:\cp^1\times \cp^1\self$, 
on the other hand, turns out to be $1$-stable as long as $P$ has no repeated roots.  To see this, one finds easily that
the irreducible components of $\cE_f$ are lines $\{y=z\}$ where $P(z)=0$, which map to points $(z,z)$ that are fixed (and not indeterminate) for $f$.  Therefore $I_{f^n}\cap I_f^- = \emptyset$ for every $n\in\N$.

The topological degree $\lambda_2(f)$ is the degree of $R(x,y)$ as a rational function of $x$.  This is actually equal to $d-1$, since the given formula for $R$ has a factor of $x-y$ in both numerator and denominator.
In particular, $f$ is not invertible as soon as $d\geq 3$.
The vector space $H^{1,1}_\R(X)$ is two dimensional, generated by the fundamental classes of generic vertical and horizontal lines $\{y = C^{st}\}$ and $\{x=C^{st}\}$.  These pull back to a vertical line and a curve of `bidegree' $(d-1,d-1)$, respectively.  Therefore, $\lambda_1(f) = \frac12\left(d-1 + \sqrt{(d-1)(d+3)}\right)$ is the largest eigenvalue of the matrix $\left[\begin{matrix} 0 & d-1 \\ 1 & d-1\end{matrix}\right]$.  Hence $f$ has small topological degree. 

\subsubsection{Blaschke products} 

A Blaschke product in two variables is a mapping of the form 
$$f(z,w) = \left( \theta_1 \prod_{i=1}^m \frac{z-a_i}{1-\overline a_i z} \prod_{i=1}^n \frac{w-b_i}{1-\overline b_i w},
 \theta_2 \prod_{i=1}^p \frac{z-c_i}{1-\overline c_i z} \prod_{i=1}^q \frac{w-d_i}{1-\overline d_i w}
 \right),$$ where $\abs{\theta_1}=\abs{\theta_2}=1$ and the complex numbers $a_1, \ldots, d_q$ have modulus less than 1. 
This class of mappings has been recently studied by Pujals and Roeder \cite{pujals-roeder}, who show in particular that 
$\lambda_1(f)$ is the spectral radius of the matrix $\left(\begin{smallmatrix} m&n \\p&q\end{smallmatrix}\right)$,
and exhibit families of Blaschke product with small topological degree.
 It is worth mentioning that this is done without constructing a 1-stable model. 
In particular it is unclear whether our results hold for these mappings.

\subsection{Maps on irrational surfaces}

Examples on irrational surfaces are much less common.  We begin by narrowing down the possibilities, showing first that maps with small topological degree do not preserve fibrations.

\begin{lem} 
\label{fibration}
Suppose that $f$ preserves a fibration $\pi:X\dashrightarrow B$ over a compact Riemann surface $B$.  
Then 
$$
\lambda_1(f) = \max\{d_g,d\} \leq d\cdot d_g = \lambda_2(f).
$$ 
where $d_g$ is the degree of the induced map $g:B\to B$ and $d$ is degree of the restriction $f:F\to f(F)$ to a generic fiber of $\pi$.
\end{lem}

\begin{proof}
The dynamical degrees are bimeromorphic invariants, so we may assume by blowing up points on $X$ that $\pi$ is holomorphic.  Since $F^2 = 0$ and $\alpha^+$ is nef, it follows from the Hodge index theorem that either $\alpha^+$ is a positive multiple of $F$ or $\pair{\alpha^+}{F} > 0$.

In the first case, we have $(f^n)^*\alpha^+ = d_g^n\alpha^+$ for all $n\in\N$.  Thus $\lambda_1(f) = d_g \leq \lambda_2(f)$.
In the second case, since $F$ is disjoint from $I_{f^n}$ for all $n\in\N$, we have
$$
\pair{(f^n)^*\alpha^+}{F} = \pair{\alpha^+}{(f^n)_*F} = \pair{\alpha^+}{(f_*)^nF} = d^n \pair{\alpha^+}{F}.
$$
Hence $\lambda_1(f) = d$.  Finally, for a generic fiber $F$ and generic $p\in F$, the inverse image $f^{-1}(F)$ has $d_g$ irreducible components and $f^{-1}(p)$ contains $d$ points in each.  Thus, $\lambda_2(f) = d\cdot d_g$. 
\end{proof}

Next we extend results of \cite{Ca, DF}.

\begin{thm}
Let $X$ be a compact K\"ahler surface and $f:X\self$ a meromorphic map of small topological degree.  Then either $X$ is rational or $\kod(X) = 0$.
\end{thm}

\begin{proof}
Recall that the Kodaira dimension $\kod(X)$ of $X$ is the dimension, for large $m$, of the image $\Phi_m(X)\subset \cp^{N_m}$ of $X$ under the `pluricanonical map' $\Phi_m:X\to\cp^{N_m}$ determined by sections of the $m$th power of the canonical bundle $K_X$ on $X$.  Necessarily, $f$ preserves the fibers of $\Phi_m$.   Indeed, if $s$ is a holomorphic section of $K_X^m$, then $f^*s$ is a meromorphic section of the same bundle, holomorphic away from $I_f$.  By Hartog's Theorem, it follows that $f^*s$ is holomorphic on all of $X$.

So if $\kod(X) =1$, it follows immediately from Lemma \ref{fibration} that $X$ does not support maps of small topological degree.  If $\kod(X) = -\infty$ and $X$ is irrational, then we can \cite[page 244]{BPV} apply a bimeromorphic transformation to assume that $X = \cp^1\times B$ for some compact curve $B$ with positive genus.  In this case the projection of $X$ onto $B$ is the Albanese fibration 
\cite[page 46]{BPV}, which must also be preserved by $f$.  Lemma \ref{fibration} again implies that $f$ cannot have small topological degree.

Suppose finally that $\kod(X) = 2$. In this case, $f$ induces a linear map $f^*:H^0(X,K_X^m) \to H^0(X,K_X^m)$ for all $m$, and for
$m>>0$, the restriction of (the projectivization of) this map to the image $\Phi_m(X)$ is bimeromorphically conjugate to $f$.  Thus 
$\lambda_2(f) = \lambda_1(f) = 1$.
\end{proof}

Since $\lambda_1(f)$ and $\lambda_2(f)$ are invariant under birational conjugacy, the next result combined with the previous one, allows us to limit attention to \emph{minimal} irrational surfaces.

\begin{prop} If $f:X\to X$ is a meromorphic and $X$ is a minimal surface with $\kod(X) = 0$, then $f$ is 1-stable.
\end{prop}

\begin{proof}
It follows from the  classification of compact complex surfaces \cite[page 244]{BPV} that $12 K_X = 0$.  Therefore from Hurwitz formula, we find that the critical divisor $C_f = K_X - f^*K_X = 0$ vanishes.  In particular $\cE_f = \emptyset$, and according to the usual criterion \cite{FoSi} we have that $f$ is 1-stable.
\end{proof}

Surface classification tells us that a minimal surface $X$ with $\kod(X) =0$ is a torus, a K3 surface, or a finite quotient of one of these. The so-called `covering trick' implies  that a meromorphic map on the base surface is necessarily induced by a map on the cover (see e.g. \cite{CantatKummer}).  So we need look only at the case of tori
and K3 surfaces.

\subsubsection{Examples on Tori.}
If $X = \C^2 / \Lambda$ is a torus, then every meromorphic map $f:X\self$ is holomorphic, and more specifically, induced by an affine map $F(z) = Az + v$ of $\C^2$ for which the lattice $\Lambda$ is forward invariant.  Since $f^*(dz_1\wedge dz_2) = (\det A)\, dz_1\wedge dz_2$, we have that $\lambda_2(f) = |\det A|^2$.  The closed $(1,1)$ forms on $X$ are generated by wedge products $dz_i\wedge d\bar z_j$, $1\leq i,j\leq 2$ of (global) $(1,0)$ and $(0,1)$ forms.   Hence $\lambda_1(f) = |r_1(A)|^2$ is the square of magnitude of the spectral radius of $A$.  For generic lattices $\Lambda$, the only possibilities for $A$ are $k\cdot I$ for some $k\in\Z$, and for $f$ arising from such $A$, one therefore has $\lambda_1(f) = k^2 \leq k^4 = \lambda_2(f)$.

Therefore tori which admit meromorphic maps with small topological degree are rare.  E. Ghys and A. Verjovsky \cite{GV} have classified the examples with $\lambda_1(f) > \lambda_2(f) = 1$.  Here we provide an elementary non-invertible example.

\begin{eg}
Let $\Lambda = \Z[i]\times \Z[i]$.  Consider $A = \left[\begin{matrix} 0 & 1 \\ 2 & d\end{matrix}\right]$, where $d\geq 2$.  Then
$A\Lambda \subset\Lambda$ and the map $f$ induced by $z\mapsto Az$ satisfies 
$$
\lambda_1(f) = \left(\frac{d+\sqrt{d^2 + 8}}{2}\right)^2 > 4 = \lambda_2(f).
$$
\end{eg}

Such examples are  Anosov, and Lebesgue measure $\mu_f$ is the unique invariant measure of maximal entropy.  Since $\det Df = A$ is constant, the Lyapunov exponents with respect to $\mu_f$ satisfy
$$
\chi^+(\mu_f) = \frac12\log\lambda_1(f) \text{ and } \chi^-(\mu_f) = -\frac12\log \lambda_1(f)/\lambda_2(f).
$$
We will show in \cite{part3} that given $\lambda_1$ and $\lambda_2$, these exponents are as small as generally possible.

\medskip

\subsubsection{Examples on K3 surfaces.}
A bimeromorphic self-map of a K3 surface is automatically an automorphism, because the absence of exceptional curves for $f^{-1}$ implies that $f$ has no points of indeterminacy and vice versa.  Cantat \cite{Ca} and \cite{Mc} have given several dynamically interesting examples of K3 automorphisms. 

On the other hand since K3 surfaces are simply connected, there are no non-invertible holomorphic maps of K3 surfaces with $\lambda_2\geq 2$.  There are nevertheless some meromorphic examples.  For instance, if $g:\C^2/\Lambda\self$ is a meromorphic map of a torus satisfying $\lambda_1(g) > \lambda_2(g)$, then one obtains \cite[\S V.16]{BPV} a quotient K3 (i.e. Kummer) surface from $\C^2/\Lambda$ by identifying points $z\mapsto -z$ and desingularizing.  The map $g$ descends to a map $f:X\self$ with $\lambda_j(f) = \lambda_j(g)$, $j=1,2$.  Observe that $I_f$ is the set of points mapped by $f$ to one of the sixteen `nodal' curves that result from desingularizing and $f$ maps each nodal curve to another nodal curve.  In particular, one can verify that a given map of a Kummer surface does not similarly descend from a torus map by checking that the set of nodal curves is not forward invariant.

\begin{eg}
\label{K3}
 If $S$ is a Riemann surface of genus two, then the Jacobian of $S$ is a two-dimensional complex torus $A$.  We let $X$ be, as above, the associated Kummer surface.  Fixing any non-invertible map $g:A\self$ (e.g. multiplication by two), we have as before an induced non-invertible meromorphic map $h:X\self$.  J. Keum \cite{Ke} has shown that for generic $S$ there exist automorphisms $\psi:X\self$ that do not preserve the set of nodal curves. Composing with an automorphism coming from the torus if necessary, we can assume that $\lambda_1(\psi)>1$.
 Therefore, $f\eqdef \psi^p\circ h^q$ is a `non-toroidal' non-invertible map as soon as $p,q\geq 1$.  Clearly $\lambda_2(f) = \lambda_2(h)^q$ and $\lambda_1(f) \leq \lambda_1(\psi)^p\lambda_1(h)^q$.  

If $\lambda_2(h) > \lambda_1(h)$, it follows for $q>>p$ that $\lambda_2(f) > \lambda_1(f)$, too.  If $q<<p$, then we claim conversely
that $f$ has small topological degree\footnote{or rather, in this case, large 1st dynamical degree!}.  To see this, note that 
by the Hodge Index Theorem, if $\alpha$ and $\beta$ are nef classes such that $\alpha^2\geq 0$ and $\beta^2\geq 0$, we have that 
 $\pair{\alpha}{\beta}^2 \geq \alpha^2\beta^2$.  So taking advantage of the fact that pullback and pushforward preserve nef classes, we estimate
\begin{eqnarray*}
\left(\int f^{n*}\omega_X\wedge \omega_X\right)^2
            & = & \left(\int \psi^{pn*}\omega_X\wedge h^{qn}_*\omega_X\right)^2 
            \geq \int (\psi^{pn*}\omega_X)^2 \int (h^{qn}_*\omega_X)^2 \\
            & \geq & \int (\psi^{pn*}\omega_X)^2 \int h^{qn}_*(\omega_X^2) 
              \geq C\lambda_1(\psi)^{2np},
\end{eqnarray*}
for some $C>0$ and all $n\in\N$.  Taking $n$th roots and letting $n\to\infty$ proves that 
$\lambda_1(f) \geq \lambda_1(\psi)^p$.  Hence $\lambda_1(f)>\lambda_2(f)$ for $q$ fixed and $p$ large enough.
\end{eg}

\subsection{Further properties of maps on irrational surfaces.}

For the remainder of this section, we assume that $f:X\to X$ is a meromorphic map with small topological degree on a torus or a K3 surface $X$.
In both cases, there is a holomorphic two form $\eta$ on $X$ that is unique once it is scaled so that the associated volume form $\nu \eqdef i\eta\wedge\bar\eta$ has unit total volume.  We have

\begin{prop}
\label{2form}
We have $f^*\eta = t\eta$ where $|t|^2 = \lambda_2(f)$.  Hence $\nu$ is an invariant probability measure with constant Jacobian $\lambda_2(f)$.
\end{prop}

Proposition \ref{2form} implies that $\norm{Df^n(x)}$ is unbounded as $n\to\infty$ at every point $x\in X$. 
Recall that the \emph{Fatou set} of a meromorphic map is the largest open set on which its iterates form a normal family.   

\begin{cor} 
If $\lambda_2(f) \geq 2$, then the Fatou set of $f$ is empty.  That is, there is no open set on which iterates of $f$ form a normal family.
\end{cor}

Since $f$ has constant Jacobian with respect to the reference measure $\nu$ we get:

\begin{cor}
\label{lyaps}
If $\mu$ is any invariant probability measure on $X$ such that $\log\norm{Df}$ is $\mu$-integrable, then the 
Lyapunov exponents $\chi^-(\mu)\leq \chi^+(\mu)$ satisfy
$$
\chi^-(\mu) + \chi^+(\mu) = \frac12\log\lambda_2(f).
$$
\end{cor}

\begin{proof}
Proposition \ref{2form} implies that for every $x$ 
$$
\frac1n \log\det Df^n(x) \approx \log\lambda_2(f).
$$
Each Lyapunov exponent has real multiplicity $2$ for $f$, so the Oseledec theorem tells us that the left side of this 
inequality tends to $2\chi^+(\mu_f) + 2\chi^-(\mu_f)$ as $n\to\infty$.
\end{proof}

We explained earlier that one never has positive closed invariant currents $T = \lambda_1^{-1} f^* T$ that are more regular than $T^+$.
It is natural to wonder whether there are other $f^*$-invariant currents at all. This is an 
interesting and delicate problem in general (see \cite{FJ1, G2} and the references therein).
Thanks to the invariant holomorphic 2-form, the answer is straightforward in the present context.
Namely we have the following result:

\begin{thm}[see also Lemma 2.7 in \cite{CantatKummer}]
Assume $X$ is a minimal surface of Kodaira dimension zero.  Let $S$ be any 
positive closed (1,1)-current on $X$. Then
$$
\frac{1}{\lambda_1^n} f^{n*} S \rightarrow c T^+\text{, with } 
c=\set{S} \cdot \alpha^-.
$$
In particular $T^+$ is the only $f^*$-invariant current.
\end{thm}

\begin{proof} 
The proof is very similar to that of Theorem \ref{thm:cvkahlerpullb},
except that we compute volumes with respect to the invariant  volume form.
Observe that $f_* \1_A \leq \lambda_2(f) \1_{f(A)}$ for any Borel set $A$. Therefore
$
\mathrm{Vol}(f^n(A)) \geq \lambda_2 (f)^{-n} \langle f^n_* \1_{A},\nu \rangle
=\mathrm{Vol}(A).
$
With these very strong volume estimates at hand, we conclude easily.
\end{proof}

The forward invariant current also behaves better.

\begin{prop}\label{prop:kodairaT-}
If $\kod(X) =0$ and $f$ has small topological degree, then $T^-$ has continuous potentials.
\end{prop}

\begin{proof} 
Since $f$ is non ramified, $f_*$ sends continuous functions to continuous functions. In particular, $\gamma^-$ is 
continuous.  Since $\lambda_1 >\lambda_2$,  the sequence $\sum_{j \geq 0} \lambda_1^{-j} (f^j)_* \gamma^-$  converges uniformly on $X$.
\end{proof}

As we will explore further in \cite{part2}, it follows that the wedge product $\mu = T^+\wedge T^-$ is a well-defined invariant probability measure, which is also the `geometric product' of the laminar/woven currents $T^{+/-}$.

\section{When $(\alpha^+)^2=0$}
\label{sec:self}

We have seen above that things can be more complicated when the invariant class $\alpha^+$ lies in the boundary of the nef cone.  In this section, we explore the extreme version of this phenomena that occurs when the self-intersection $(\alpha^+)^2$ vanishes.  This was done for bimeromorphic maps $f:X\self$ in \cite{DF}, where it was proved that $(\alpha^+)^2 = 0$ implies that $f$ is bimeromorphically conjugate to an automorphism of a (smooth) surface.  Here we obtain a similar result, except that the new surface can be singular.  We are grateful to Charles
Favre for many helpful comments about this section and, as a postscript, we refer the reader to his recent preprint \cite{Fav08} extending the results we discuss here.

Note that, besides the vanishing of $(\alpha^+)^2$ our main assumption throughout this section is that 
$\lambda_1^2 > \lambda_2$.  We do not assume that $f$ has small topological degree or even, until the end, that $f$ is $1$-stable.  Hence it is necessary at the beginning to allow that the spectral radius $r_1$ for $f^*$ might be larger than $\lambda_1$.  We then have $f^*\alpha^+ = r_1\alpha^+$.

We will need the following consequence of the pushpull formula from \cite{DF}.

\begin{prop}
\label{pushpull1}
There exists a non-negative quadratic form $Q$ on $\HR(X)$ such that for all $\alpha,\beta\in \HR(X)$,
$$
\pair{f^*\alpha}{f^*\beta} = \lambda_2\pair{\alpha}{\beta} + Q(\alpha,\beta).
$$
Moreover $Q(\alpha,\alpha) = 0$ if and only if $\pair{\alpha}{C} = 0$ for every irreducible $C\subset\cE^-_f$.
\end{prop}

Our next result will allow us to ignore the distinction between $\lambda_1$ and $r_1$ when $(\alpha^+)^2=0$.  

\begin{prop}
\label{zero}
Suppose that $r_1^2 > \lambda_2$.  Then the following are equivalent.
\begin{enumerate}
\item $(\alpha^+)^2 = 0$.
\item $\pair{\alpha^+}{C} = 0$ for every $C\in \cE^-_f$.
\item $f_*\alpha^+ = \frac{\lambda_2}{r_1} \alpha^+$.
\item For any proper modification $\pi:\hat X \to X$, we have $\hat f^*\hat\alpha^+ = r_1 \hat\alpha^+$, where $\hat\alpha^+ = \pi^*\alpha^+$ and $\hat f$ is the map induced by $f$ on $\hat X$.
\end{enumerate}
In any case, we have that $(f^*)^n\alpha^+ = (f^n)^*\alpha^+$ for every $n\in\N$.  In particular, $\lambda_1 = r_1$.
\end{prop}

\noindent 
The condition (4) gives, in the case where $X$ is rational, the connection with \cite{BFJ}.  In that paper, the authors prove existence of an invariant class $\tilde\alpha^+$ for the action of $\lambda_1^{-1}f^*$ on `$L^2$-cohomology classes of the Riemann-Zariski space'.  There is a natural projection of $\tilde\alpha^+$ into $H^{1,1}_\R(X)$ and indeed into $H^{1,1}_\R(\hat X)$ where $\pi:\hat X \to X$ is proper modification of $X$.  In general, however, the image of $\tilde\alpha^+$ is not (a multiple of) $\alpha^+$.  Condition (4) says that the image actually is $\alpha^+$ and that, more generally, the projection of $\tilde\alpha^+$ into $H^{1,1}_\R(\hat X)$ is $\pi^*\alpha^+$ for any modification $\hat X$.  In the language of \cite{BFJ}, `$\tilde\alpha^+$ is Cartier and determined in $X$'.  The other results in this section confirm to some extent the expectation \cite[Remark 3.9]{BFJ} that meromorphic maps with  Cartier eigenclasses should have some rigidity properties.

Though we give a separate proof here, the final conclusion may also be seen to proceed more or less immediately from (4) and the work in \cite{BFJ}.

\begin{proof}
Proposition \ref{pushpull1} and invariance of $\alpha^+$ tell us that
$$
r_1^2 (\alpha^+)^2 = (f^*\alpha^+)^2 \geq \lambda_2(\alpha^+)^2
$$ 
with equality if and only if $\pair{\alpha^+}{C} = 0$ for all curves $C\subset \cE^-$.  Since by assumption $r_1^2 > \lambda_2$, this gives the equivalence of (1) and (2).  The equivalence of the (2) and (3) follows from Proposition \ref{pushpull2} (for cohomology classes): 
$$
r_1 f_*\alpha^+ = f_*f^*\alpha^+ = \lambda_2 \alpha^+ + E^-(\alpha^+),
$$
where $E^-(\alpha^+) = 0$ if and only if $\pair{\alpha^+}{C}=0$ for every component $C$ of $\cE^-_f$.

To prove equivalence of (3) and (4), we begin with the equality $f^* = \pi_*\hat f^*\pi^*$ (this follows from e.g. \cite[Proposition 1.13]{DF}).  Hence by Proposition \ref{pushpull2}
$$
r_1\hat\alpha^+ = \pi^*f^*\alpha^+ = \pi^*\pi_*\hat f^*\hat\alpha^+ = \hat f^*\alpha^+ + E,
$$
where $E$ is an effective divisor supported on $\cE_\pi$ that vanishes if and only if $\pair{\hat f^*\hat\alpha^+}{\cE_\pi} = 0$.  Since $\pair{\hat f^*\hat\alpha^+}{\cE_\pi} = \pair{\alpha^+}{\pi_*\hat \cE_\pi}$ and $\supp\hat \pi_*f_*\cE_\pi\subset \cE_f^-$, we infer that when (3) holds, $E = 0$.  That is (3) implies (4).  On the other hand, if (3) fails, then from Proposition \ref{pushpull1} and $r_1^2 > \lambda_2$, we deduce that $\pair{\alpha^+}{f(p)} > 0$ for some $p\in I_f$.  Thus (4) fails in the case where $\pi$ is the blowup of $X$ at $p$.
Equivalence of (1)-(4) is now established.

To get equivalence of $r_1$ and $\lambda_1$, let $\theta^+$ be a positive representative of $\alpha^+$ with bounded potentials.  Then $(f^n)^*\theta^+$ is a positive closed current, and the difference $(f^*)^n\theta^+ - (f^n)^*\theta^+$ is a current of integration over an effective divisor.  Hence it suffices to show that the Lelong numbers $\nu((f^*)^n\theta^+,p)$ vanish for every $n\in\N$ and $p\in X$.  We do this inductively.

The case $n=0$ is immediate.  Assuming $\nu((f^*)^{n-1}\theta^+,p) \equiv 0$ on $X$, we have from (2) in Proposition \ref{zero} that $\pair{\alpha^+}{f(p)} = 0$ for every $p\in I_f$.   Thus since $(f^*)^{n-1}\theta^+$ represents $r_1^{n-1}\alpha^+$, Proposition \ref{pullback lelong 2} tells us that the Lelong numbers of  $(f^*)^n\theta^+$ also vanish everywhere on $X$. 
\end{proof}

%
%
%
%
%
%

The following version of the next result was pointed out to us by Charles Favre.

\begin{prop} 
\label{ineffective}
Suppose that $\alpha^+$ is cohomologous to an effective divisor $D$ such that $D^2 = 0$.  Then
 $\lambda_1(f)$ and $\lambda_2/\lambda_1$ are integers.  In particular, $f$ does not have small topological degree.
\end{prop}

\begin{proof}
Suppose $\alpha^+$ is cohomologous to $D$ and $D^2 = 0$.  Let $H\subset H^{1,1}_\R(X)$ be the subspace spanned by curves $C\subset X$ with $\pair{\alpha^+}{C} = 0$.  Then by the Hodge Index Theorem, $\alpha^+$ spans the kernel of the restriction of the intersection pairing to $H$.  Since the pairing is integral, it follows that after rescaling $\alpha^+ \in H^{1,1}_\R(X) \cap H^2(X,\Z)$.   Both $f^*$ and $f_*$ preserve integral classes, so it follows that $\lambda_1$ and (by Proposition \ref{zero}) $\lambda_2/\lambda_1$ are integers.
\end{proof}

In order to state and prove the next several results we establish some useful notation.  Suppose that $S = \bigcup_{j\in\N} C_j$ is 
a countable union of irreducible curves in $X$.  The given decomposition into irreducibles is the only one possible, so it makes sense to call the curves $C_j$ `the' irreducible components of $S$ and let $\div(S)$ denote the set of all divisors of the form $\sum_{j\in\N} a_j C_j$, $a_j \neq 0$ for only finitely many $j$.  As before, we let $\norm{\cdot}$ be any norm on $\HR(X)$.  We will say that the intersection form is \emph{negative definite on $S$} if there exists $C>0$ such that $\pair{D}{D} \leq -C\norm{D}^2$ for all $D\in\div(S)$.  It is a classical observation of Zariski \cite[page 111]{BPV} that negative definiteness is implied by an apparently weaker condition.

\begin{prop}
\label{zariski}
 Suppose $D^2<0$ for every non-zero effective $D\in\div(S)$.  Then $S$ has only finitely many irreducible components, and the intersection form is negative definite on $S$.
\end{prop}

We  will apply Proposition \ref{zariski} to the sets 
$$
\cE^-_\infty = \bigcup_{n\in\N} \cE^-_{f^n},\qquad \cE^+_\infty = \bigcup_{n\in\N} \cE_{f^n}.
$$
By Proposition \ref{curves}, $\div(\cE^-_\infty)$ is $f_*$ invariant and $\div(\cE^+_\infty)$ is $f^*$ invariant.  It can happen, as on irrational surfaces that $\cE^+_\infty$ is empty while $\cE^-_\infty$ is not, but except under very special circumstances, the reverse situation never occurs:

\begin{prop} 
\label{i before e}
Suppose that either $(\alpha^+)^2>0$ or that $\alpha^+$ is not the cohomology class of an effective divisor.
Then for every $C\subset \cE_f$, there exists $n\in\N$ such that either $C$ or $f(C)$ is contained in $\supp (f_*)^n\cE^-_f$.
\end{prop}

\begin{proof} Let $C\subset \cE_f$ irreducible.
Suppose for all $n\in\N$ that $C\not\subset\cE^-_{f^n}$.  Then since $f$ is dominating, we can find a sequence of curves $C_n\subset X$ such that $C=C_0$ and for all $n\geq 0$, $f(C_{n+1}) = C_n$.  Because $f(C)$ is a point and in particular not equal to any of the curves in this sequence, we see that all the $C_n$ are distinct.

Now suppose further that $f(C)\not\subset\cE^-_{f^n}$ for any $n\in\N$.  Then since $f^{n+1}_* C_n$ is a divisor with connected support containing the point $f(C)$ and contained in $f^{n+1}(C_n)\cup \cE^-_{f^n} = f(C) \cup \cE^-_{f^n}$, we deduce that $(f_*)^{n+1} C_n = 0$
for every $n\in\N$.  Hence 
$$
\lambda_1^{n+1}\pair{\alpha^+}{C_n} = \pair{(f^*)^{n+1}\alpha^+}{C_n}= \pair{\alpha^+}{(f_*)^{n+1} C_n} = 0.
$$
for all $n\in\N$.  The hypothesis of the proposition and the Hodge Index Theorem imply that $C_n^2 < 0$ for
all $n$.  There are infinitely many $C_n$, so this contradicts Proposition \ref{zariski}.
\end{proof}

From Proposition \ref{curves} we immediately get

\begin{cor}
\label{i before e again}
Let $D\subset \div(\cE^+_\infty)$ be given.  Then $\supp (f_*)^n D\subset \cE^-_\infty\cup \cE^+_\infty$ for all $n\in\N$.
Under the hypotheses of Proposition \ref{i before e}, we have further that there exists a curve $C\subset \cE^-_\infty$ and an integer $n\in\N$ such that $\supp D \subset\supp (f^*)^n C$ 
\end{cor}

\begin{thm}
\label{holoize}
Suppose that $\lambda_1^2 > \lambda_2$, that $(\alpha^+)^2 = 0$ and that $\alpha^+$ is not cohomologous to an effective divisor.  Then there is a modification $\pi:X\to\check X$ of a singular surface $\check X$ by $X$ that conjugates $f$ to a holomorphic map $\check f:\check X\self$.
The exceptional set of $\pi$ is $\cE^-_\infty$.
\end{thm}

We begin by showing that $\cE^-_\infty$ can be contracted.

\begin{lem}
\label{nearly holo} $\cE^-_\infty$ is a union of finitely many curves on which the intersection form is negative definite.  Moreover, $\div(\cE^-_\infty)$ is invariant under both $f^*$ and $f_*$ and we have in particular that $\cE^+_\infty\subset \cE^-_\infty$. 
\end{lem}

\begin{proof}
If $(\alpha^+)^2 = 0$, then by Proposition \ref{zero} we have $\div(\cE^-_f) \subset (\alpha^+)^\perp$.  Moreover, both $f^*\alpha^+$ and $f_*\alpha^+$ are proportional to $\alpha^+$, so we have further that $\div(S)\subset (\alpha^+)^\perp$ where 
$$
S \eqdef \bigcup_{n\in\N,\, D\in\div(\cE^-_\infty)} \supp (f^*)^n D.
$$
As before the assumptions imply that  $D^2 < 0$ for every effective $D\in\div(S)$. Hence $S$ has finitely many irreducible components and the intersection form is negative definite on $S$.

Now $\div(S)$ is $f^*$-invariant by definition, so we will be finished once we show that $S = \cE^-_\infty$.  To this end, let $C$ be
any irreducible component of $\div(S)$.  We consider two cases.  If $C = f^k(C)$ is periodic, then since $f^n(C) \subset \cE^-_\infty$ for some $n\in\N$, we deduce that $C \subset f^n(C)\cup \dots \cup f^{n+k}(C)$ is also included in $\cE^-_\infty$.  

If $C$ is not periodic, then we may consider a (maximal) sequence of curves $C_0,\dots,C_j\subset S$ defined as follows.  Taking $C_0 = C$, we choose $C_{j+1}$ to be any curve such that $f(C_{j+1}) = C_j$.  As $C$ is not periodic, we must have $C_{j+1} \neq C_i$ for any $i\leq j$.  On the other hand, $S$ has only finitely many irreducible components, so eventually we will be unable to find the desired $C_{j+1}$.  The only alternative is that $C_j\subset \cE^-_f$.  That is, $C \subset \cE^-_{f^j} \subset \cE^-_\infty$.  Thus $S = \cE^-_\infty$.
\end{proof}

From Lemma \ref{nearly holo}, we complete the proof of Theorem \ref{holoize} as follows.  Since the intersection form is negative definite on $\cE^-_\infty$, and since this set has finitely many irreducible components,  a criterion of Grauert \cite[page 91]{BPV} implies then that there is a bimeromorphic morphism $\check\pi:X\to\check X$ of $X$ onto a singular surface $\check{X}$ with exceptional set $\cE_\pi = \cE^-_\infty$.  Each connected component of $\cE^-_\infty$ maps to a distinct point of $\check{X}$. 

Suppose now that $\check C \subset \check X$ is an irreducible curve with $\check f(C)$ a point.  Then $C = \pi(C')$ for some
irreducible $C'\subset X$ and $\pi(f(C')) = \check f(C)$.  If $f(C')$ is a point, then $C'\subset\cE_f$ and by Theorem \ref{nearly holo}
$C'\subset\cE^-_\infty$.  Hence $\pi(C')$ is a point, contrary to our choice of $C'$.  If $f(C')$ is a curve, then $f(C')$ is exceptional for $\pi$ and therefore a component of $\cE^-_\infty$.  But Theorem \ref{nearly holo} also tells us that 
$\div(\cE^-_\infty)$ is $f^*$-invariant, so it follows that $C'\leq f^*f(C')$ is itself a component of $\cE^-_\infty$.  Again we are forced to conclude that $\pi(C')$ is a point rather than $C$.  It follows that $\cE_{\check f} = \emptyset$.  One shows similarly that $I_{\check{f}}$ is empty.  Thus $\check f$ is holomorphic.
\qed\medskip

It is natural to wonder what the analogues of the above results are when we work with the $f_*$-invariant class $\alpha^-$ instead of $\alpha^+$.  

\begin{prop}  Suppose that $\lambda_1^2 > \lambda_2$.  Then 
\begin{itemize}
\item $\alpha^-$ is determined in $X$ if and only if 
$\pair{\alpha^-}{\cE_f} = 0$.
\item If $\alpha^+$ is determined in $X$ and is not cohomologous to an effective divisor, then $\alpha^-$ is also determined in $X$.
\end{itemize} 
\end{prop}

\begin{proof}
The first item is proved in the same way as equivalence of (3) and (4) in Proposition \ref{zero}.  To prove the second, recall from the same Proposition that $\alpha^+$ is determined in $X$ if and only if $(\alpha^+)^2 = 0$.
Since $\alpha^+$ is not cohomologuous to an effective divisor, Lemma \ref{nearly holo} tells us that the intersection $\cE^+_\infty$ is an $f^*$-invariant subspace not containing $\alpha^+$.  Thus $\lambda_1^{-n}(f^*)^n E \to 0$ for every irreducible $E\subset \cE^+_\infty$.  From Theorem \ref{spectral} we get $\pair{\alpha^-}{\cE_f} = 0$.
\end{proof}

It does not seem likely that determination of $\alpha^-$ in $X$ is equivalent to $(\alpha^-)^2 = 0$.  There is however an implication in one direction.

\begin{prop}
Suppose that $\lambda_1^2 > \lambda_2$.  
\begin{itemize}
 \item $(\alpha^-)^2 = 0$ if and only if $f^*\alpha^- = \frac{\lambda_2}{\lambda_1}\alpha^-$.
 \item If $(\alpha^-)^2 = 0$ and $\alpha^-$ is not the class of an effective divisor, then $(\alpha^+)^2 = 0$.
\end{itemize}
\end{prop}

Observe that if $\alpha^-$ is represented by an effective divisor, then one shows as in Proposition \ref{ineffective} that $\lambda_1$ and $\lambda_2/\lambda_1$ are integers.

\begin{proof}
If $f^*\alpha^- = \frac{\lambda_2}{\lambda_1}\alpha^-$, then one shows $(\alpha^-)^2 = 0$ as in the proof that (2) implies (1) in Proposition \ref{zero}.  For the reverse implication, we observe that if $(\alpha^-)^2=0$, then $\pair{f^*\alpha^-}{\alpha^-} = \pair{\alpha^-}{f_*\alpha^-} = 0$.  The Hodge Index Theorem implies $f^*\alpha^-= t\alpha^-$ for some $t\geq 0$.  From Proposition \ref{pushpull2}, we deduce that (on the level of cohomology)
$$
(t\lambda_1 - \lambda_2) \alpha^- = E^-(\alpha^-).
$$
Thus $\pair{\alpha^-}{E^-(\alpha^-)} = 0$, which according to the characterization of $\supp E^-(\alpha^-)$ in Proposition \ref{pushpull2} implies $E^-(\alpha^-) = 0.$  We conclude that $t = \lambda_2/\lambda_1$.

Continuing with the assumption $(\alpha^-)^2 = 0$, we further have from Proposition \ref{pushpull2} that $\pair{\alpha^-}{\cE_f^-} = 0$.  Since $f^*\alpha^-$ is a multiple of $\alpha^-$, we obtain more generally that
$\pair{\alpha^-}{E} = 0$ for every curve $E\subset \cE_\infty^-$.  If $\alpha^-$ is not cohomologous to an effective divisor, then we have from Proposition \ref{zariski} that the intersection form is negative on $\cE_\infty^-$.  Since
$\cE_\infty^-$ is $f_*$-invariant, we conclude from Theorem \ref{spectral} that $\pair{\alpha^+}{E} = 0$ for every
$E\subset\cE_\infty^-$, and in particular for every $E\subset\cE_f^-$.  Thus $(\alpha^+)^2 = 0$.
\end{proof}

It is well known that if $f$ is holomorphic, i.e. $I_f=\emptyset$, then $T^+$ has continuous potentials (see \cite{Sib}).  We end with the observation that $T^+$ and $T^-$ are similarly well-behaved when their self-intersections vanish.

\begin{thm}\label{pro:g+}
Suppose $f$ is $1$-stable and has small topological degree.  If  $(\alpha^+)^2=0$, then $T^+$ has bounded potentials.  If  $(\alpha^-)^2 =0$, then  both $T^+$ and $T^-$ have bounded potentials.
\end{thm}

\begin{proof}
Let $\theta^+, \gamma^+, g^+$ be as in the proof of Theorem \ref{thm:cv+}.  When $(\alpha^+)^2 = 0$, we have $\pair{\alpha^+}{C} = 0$ for every $\cE^-_f$.  So applying Proposition \ref{negative2} to $\pm\theta^+$ tells us that $\gamma^+$ is bounded.  From this it is easy to see that the sequence defining $g^+$ is uniformly convergent on $X$.  Thus $T^+$ has a bounded potentials.  When $(\alpha^-)^2 = 0$, the reasoning is similar for $T^-$, 
and $(\alpha^-)^2 = 0$ implies $(\alpha^+)^2 = 0$.
\end{proof}

\bibliographystyle{alpha}
\bibliography{refs3}
\end{document}